\documentclass{amsart}% option [reqno] puts equat numb to the right

\usepackage[mathscr]{eucal}% So-called Euler fonts, allows \mathscr
\usepackage{amssymb}% allows special arrows like >-> e.g.
\usepackage{amsthm}% allows theorem* , e.g.
\usepackage{bbold}% allows \unit \mathbb{1}
\usepackage{enumerate}% allows easy change of label
\usepackage{array}% allows array
\usepackage{amsmath}% allows pmatrix

\usepackage[colorlinks=true,linkcolor={black},citecolor={black},pdfstartview={FitH}]{hyperref}

\numberwithin{equation}{section}% makes equat numb contain the section
\usepackage[all]{xy}
\newdir{ >}{{}*!/-10pt/\dir{>}}
\swapnumbers %pour que le numéro apparaisse devant le théorème

\newtheorem{Thm}[equation]{Theorem}
\newtheorem*{Thm*}{Theorem}
\newtheorem{Prop}[equation]{Proposition}
\newtheorem{Lem}[equation]{Lemma}
\newtheorem{Cor}[equation]{Corollary}
\theoremstyle{remark}
\newtheorem{Def}[equation]{Definition}
\newtheorem*{Def*}{Definition}
\newtheorem{Not}[equation]{Notation}
\newtheorem*{Not*}{Notation}
\newtheorem{Exa}[equation]{Example}
\newtheorem{Rem}[equation]{Remark}
\newtheorem{Que}[equation]{Question}
\newcommand{\nc}{\newcommand}
\nc{\dmo}{\DeclareMathOperator}
\nc{\Displ}{\displaystyle}

\nc{\bbA}{\mathbb{A}}
\nc{\bbB}{\mathbb{B}}
\nc{\bbG}{\mathbb{G}}
\nc{\bbN}{\mathbb{N}}
\nc{\bbQ}{\mathbb{Q}}
\nc{\bbZ}{{\mathbb{Z}}}

\nc{\gp}{\mathfrak{p}}% prime p
\nc{\eg}{{e.g.}}
\nc{\ie}{{i.e.}\ }
\nc{\aka}{{a.\,k.\,a.}\ }

\dmo{\Ab}{Ab}% cat of abelian groups
\dmo{\Add}{Add}% additive categories
\dmo{\Aut}{Aut}%
\dmo{\Cat}{Cat}% categories
\dmo{\chara}{char}% characteristic
\dmo{\CoInd}{CoInd}% coinduction
\dmo{\Coker}{Coker}
\dmo{\Comod}{Comod}% comodules
\dmo{\cone}{cone}
\dmo{\Der}{D}% ground notation for derived categories
\dmo{\Desc}{Desc}% descent category
\dmo{\End}{End}
\dmo{\Free}{Free}
\dmo{\Hom}{Hom}
\dmo{\id}{id}
\dmo{\Id}{Id}
\dmo{\im}{im}
\dmo{\Img}{Im}
\dmo{\Ind}{Ind}% induction
\dmo{\Ker}{Ker}
\dmo{\Komp}{K}% ground notation for htpy categories
\dmo{\Mod}{Mod}% modules
\dmo{\Mor}{Mor}%
\dmo{\modname}{mod}%
\dmo{\Or}{Or}
\dmo{\Pic}{Pic}
\dmo{\pr}{pr}
\dmo{\Proj}{Proj}
\dmo{\Qcoh}{Qcoh}% quasi-coherent modules
\dmo{\rmB}{B}
\dmo{\rmH}{H}
\dmo{\Rep}{Rep}
\dmo{\Res}{Res}
\dmo{\rmL}{L}
\dmo{\rmR}{R}
\dmo{\setsname}{sets}
\dmo{\Sets}{Sets}
\dmo{\smallb}{b}% ground exponent for ``bounded''
\dmo{\smallperf}{perf}% ground exponent for ``perfect''
\dmo{\stab}{stab}% stable category of fin. gen. mod.
\dmo{\supp}{supp}
\dmo{\SH}{SH}% ground name for cat of spectra
\dmo{\Spc}{Spc}
\dmo{\Spec}{Spec}
\dmo{\St}{St}% stable category of used with \Rep
\dmo{\Stab}{Stab}% stable category of non-fin. gen. mod.
\dmo{\Vect}{Vect}% category of vector spaces

\nc{\adh}[1]{\overline{#1}}% adherence
\nc{\adj}{\leftrightarrows}
\nc{\begen}{\begin{enumerate}}
\nc{\bfA}{\mathbf{A}}
\nc{\bsl}{\backslash}
\nc{\cat}[1]{\mathscr{#1}}%or: \nc{\cat}[1]{\mathcal{#1}}
\nc{\Cech}{\v Cech}
\nc{\cg}[1]{\,^{#1}}% (left) conjugation
\nc{\cH}{\textrm{\rm\v H}{}}
\nc{\ciso}[1]{\beta_{#1}}% conjugation iso on ring objects
\nc{\CComod}{\,\text{-}\Comod}%
\nc{\con}[2]{{{\vphantom{#2}}^{#1}\!#2}}% cong of subgroup
\nc{\cRes}[1]{\con{#1}\Res{}}
\nc{\Db}{\Der^{\smallb}}% derived bounded category
\nc{\DAC}{\Desc_{\cat C}(A)}% most used
\nc{\DACG}{\Desc_{\cat C(G)}(A)}% most used
\nc{\doublequot}[3]{#1\backslash #2/#3}% double cosets
\nc{\Dperf}{\Der^{\smallperf}}% derived category of perfect compl
\nc{\dto}{\kern-.2em\downarrow}
\nc{\Endcat}[1]{\End_{\cat #1}}
\nc{\EndK}{\Endcat{K}}%most used
\nc{\ened}{\end{enumerate}}
\nc{\eps}{\epsilon}
\nc{\equalby}[1]{\overset{\textrm{#1}}=}
\nc{\equalbyeq}[1]{\equalby{\eqref{#1}}}
\nc{\FFreecat}[1]{\,\textrm{-}\Free_{\cat #1}}
\nc{\gH}{\con{g\!}{H}}% most used
\nc{\gK}{\con{g\!}{K}}% most used
\nc{\gL}{\con{g}{L}}% most used
\nc{\Gm}{\bbG_{\textrm{\rm m}}}
\nc{\gRes}{\cRes{g}}%most used
\nc{\Gsets}{{G\textrm{-}\sets}}
\nc{\gXfunc}[1]{\con{#1}{(-)}}
\nc{\gX}[1]{\con{g}{#1}}
\nc{\HGH}{\doublequot{H}{G}{H}}
\nc{\HGL}{\doublequot HGL}% most used
\nc{\hook}{\hookrightarrow}
\nc{\Homcat}[1]{\Hom_{\cat #1}}
\nc{\HomK}{\Homcat{K}}
\nc{\ic}[1]{{#1}^{\natural}}% idemp compl of #1
\nc{\iccat}[1]{\ic{\cat #1}}% idemp compl of cat #1
\nc{\Idcat}[1]{\Id_{\cat #1}}% identity functor of cat #1
\nc{\Index}[2]{[\,#1\!:\!#2\,]}
\nc{\ihom}{{\underline{\hom}}}
\nc{\into}{\mathop{\rightarrowtail}}
\nc{\inv}{^{-1}}
\nc{\isoto}{\oto{\sim}}%{\buildrel \sim\over\to}
\nc{\isotoo}{\mathop{\otoo{\sim}}}
\nc{\Kb}{\Komp^{\smallb}}% htpy bounded category
\nc{\kk}{\Bbbk}
\nc{\LF}{\rmL\!F}
\nc{\mapstoo}{\longmapsto} \nc{\onto}{\mathop{\twoheadrightarrow}}
\nc{\matrice}[1]{\begin{pmatrix} #1 \end{pmatrix}}
\nc{\Mid}{\,\big|\,}
\nc{\mmod}{\,\text{--}\modname}%
\nc{\MMod}{\,\textrm{-}\Mod}
\nc{\MModcat}[1]{\MMod_{\cat #1}}
\nc{\mutmut}{{\sl mutatis mutandis}}
\nc{\op}{{^{\operatorname{op}}}}
\nc{\oto}[1]{\,\overset{#1}\to\,}
\nc{\otoo}[1]{\,\overset{#1}\too\,}
\nc{\ourfrac}[2]{\genfrac{}{}{0pt}{}{\scriptstyle #1}{\scriptstyle #2}}
\nc{\potimes}[1]{^{\otimes #1}}% tensor power
\nc{\pp}[1]{^{(#1)}}
\nc{\PProj}{\,\textrm{-}\Proj}
\nc{\prj}[2]{\pi_{#1,#2}}
\nc{\prc}[3]{\pi_{#1,#2,#3}}
\nc{\prHL}{\prj{H}{L}}% most used
\nc{\prHLg}{\prc{H}{L}{g}}% most used
\nc{\qquadtext}[1]{\qquad\text{#1}\qquad}
\nc{\quadtext}[1]{\quad\text{#1}\quad}
\nc{\QQcoh}[1]{\Qcohname(#1)}
\nc{\RG}{\rmR\!G}
\nc{\rh}[2]{\alpha_{#1,#2}}% ring homom from A_#2 to A_#1
\nc{\rhLH}{\rh{L}{H}}% most used
\nc{\sbull}{{\scriptscriptstyle\bullet}}%\mathbf{\cdot}}%{}}
\nc{\sbulll}{{\scriptscriptstyle\bullet\bullet}}%\mathbf{\cdot}}%{}}
\nc{\sets}{{\setsname}}%{\textsc{sets}}%
\nc{\sK}{\con{s}K}% most used
\nc{\smallcon}[2]{{{}^{#1}\!#2}}
\nc{\smalliff}{\Leftrightarrow}
\nc{\smallmatrice}[1]{\left(\begin{smallmatrix} #1 \end{smallmatrix}\right)}
\nc{\sRes}{\cRes{s}{}}% most used
\nc{\ssi}{\Leftrightarrow}
\nc{\sstab}{\,\textrm{-}\stab}%
\nc{\SStab}{\,\textrm{-}\Stab}%
\nc{\SET}[2]{\big\{\,#1\Mid#2\,\big\}}
\nc{\SHfin}{\SH^{\text{\rm fin}}}% stab. hom. cat of finite spectra
\nc{\siff}{\Leftrightarrow}% small iff
\nc{\StPic}{\Pic^{\scriptscriptstyle\textrm{\,\rm st}}}
\nc{\StRep}{\Stab\Rep}
\nc{\tK}{\con{t}{K}}% most used
\nc{\tL}{\con{t}{L}}% most used
\nc{\too}{\mathop{\longrightarrow}\limits}
\nc{\tooo}[1]{\mathop{\vcenter{\hbox to #1em{\hrulefill}}\kern-5pt\to}\limits}
\nc{\tRes}{\cRes{t}}% most used
\nc{\uA}{\underline{A}}
\nc{\uK}{\con{u}{K}}% most used
\nc{\unit}{\mathbb{1}}%{1\!\!1}}% unit for \otimes
\nc{\cC}{\check C}
\nc{\VVect}{\,\textrm{-}\Vect}
\nc{\xyptriangle}[7]{\xymatrix@C=#7em{{#1} \ar[r]^-{\Displ #4} & {#2} \ar[r]^-{\Displ #5}&{#3}\ar[r]^-{\Displ #6}&\Sigma #1\,.}}
\nc{\xystriangle}[7]{\xymatrix@C=#7em{{#1}\ar[r]^-{#4} & {#2} \ar[r]^-{#5}&{#3}\ar[r]^-{#6}&\Sigma {#1}}}% small version of the above
\nc{\xysTriangle}[8]{\xymatrix@C=#8em{{#1}\ar[r]^-{#5}&{#2}\ar[r]^-{#6}&{#3}\ar[r]^-{#7}&{#4}}}% small version of the above
\nc{\xytriangle}[7]{\xymatrix@C=#7em{{#1} \ar[r]^-{\Displ #4} & {#2} \ar[r]^-{\Displ #5}&{#3}\ar[r]^-{\Displ #6}&\Sigma #1}}
\nc{\xyTriangle}[8] {\xymatrix@C=#8em{#1\ar[r]^-{\Displ #5}&#2\ar[r]^-{\Displ #6}&#3\ar[r]^-{\Displ #7}&#4}}

\begin{document}

%------------------------------------------------------------------------------

\title[Stacks of group representations]{Stacks of group representations}
\author{Paul Balmer}
\date{2012 August 1}

\address{Paul Balmer, Mathematics Department, UCLA, Los Angeles, CA 90095-1555, USA}
\email{balmer@math.ucla.edu}
\urladdr{http://www.math.ucla.edu/$\sim$balmer}

\begin{abstract}
We start with a small paradigm shift about group representations, namely the observation that restriction to a subgroup can be understood as an extension-of-scalars. We deduce that, given a group~$G$, the derived and the stable categories of representations of a subgroup $H$ can be constructed out of the corresponding category for~$G$ by a purely triangulated-categorical construction, analogous to \'etale extension in algebraic geometry.

In the case of finite groups, we then use descent methods to investigate when modular representations of the subgroup $H$ can be extended to~$G$. We show that the presheaves of plain, derived and stable representations all form stacks on the category of finite $G$-sets (or the orbit category of~$G$), with respect to a suitable Grothendieck topology that we call the \emph{sipp topology}.

When $H$ contains a Sylow subgroup of~$G$, we use sipp \Cech\ cohomology to describe the kernel and the image of the homomorphism $T(G)\to T(H)$, where $T(-)$ denotes the group of endotrivial representations.
\end{abstract}


\thanks{Research supported by NSF grant~DMS-0969644.}

\maketitle

\vskip-\baselineskip\vskip-\baselineskip
\tableofcontents

%------------------------------------------------------------------------------
\goodbreak
\section{Introduction}
\bigbreak
%------------------------------------------------------------------------------

For the whole paper, $G$ is a group, $\kk$ a commutative ring and $p$ a prime number.
\begin{Not}
\label{not:123}%
We denote by $\cat C(G)$ either one of the following categories\,:
\begin{enumerate}[(1)]
\item\label{it:1}%
\quad $\cat C(G)= \kk G\MMod$ \quad the category of left $\kk G$-modules, \underline{or}
\smallbreak
\item\label{it:2}%
\quad $\cat C(G)= \Der(\kk G)$ \kern 2.2em its derived category, assuming $\kk$ a field, \underline{or}
\smallbreak
\item\label{it:3}%
\quad $\cat C(G)= \kk G\SStab$ \quad its stable category, assuming $\kk$ a field and $G$ finite.
\end{enumerate}
\end{Not}

\goodbreak

Let $H\leq G$ be a subgroup. Our initial observation is that, in all three cases, \emph{restriction $\Res^G_H:\cat C(G)\to \cat C(H)$ is an extension-of-scalars\,!} This slogan seems ludicrous at first sight but makes sense if we understand ``extension-of-scalars" in the appropriate way. In this vein, Theorems~\ref{thm:fin-ind} and~\ref{thm:f-i-der} give us\,:

\begin{Thm}
\label{thm:main}%
Suppose that $H\leq G$ is a subgroup of finite index. Then there exists a commutative separable ring object $A=A^G_H$ in the symmetric monoidal category~$\cat C(G)$ and a canonical equivalence $\Psi=\Psi^G_H: \ \cat C(H)\isotoo A\MModcat{C(G)}$ between the category $\cat C(H)$ and the category $A\MModcat{C(G)}$ of $A$-modules \underbar{in~$\cat C(G)$}, under which the restriction functor $\Res^G_H$ becomes isomorphic to the extension-of-scalars functor~$F_A=A\otimes-$ with respect to~$A$, \ie the following diagram
$$
\xymatrix@C=1em@R=2em{
& \cat C(G) \ar[ld]_(.6){\Displ \Res^G_H} \ar[rd]^(.6){\ \ \Displ F_{A}}
\\
\cat C(H) \ar[rr]_-{\cong}^-{\Displ \Psi}
&& A\MModcat{C(G)}
}
$$
commutes up to isomorphism. Under this equivalence, (co)induction $\cat C(H)\too \cat C(G)$ becomes isomorphic to the functor $A\MModcat{C(G)}\too\cat C(G)$ which forgets $A$-actions. Explicitly, the ring object $A^G_H$ is the usual $\kk G$-module $\kk(G/H)$ with multiplication given by $\gamma\cdot\gamma=\gamma$ and $\gamma\cdot\gamma'=0$ for every $\gamma\neq\gamma'$ in~$G/H$ (see Definition~\ref{def:AH}).
\end{Thm}

This result relies in an essential way on the use of $A$-modules \emph{in the category~$\cat C(G)$}, \`a la Eilenberg-Moore~\cite{EilenbergMoore65}; see Section~\ref{se:monadicity}. This half-a-century old concept of modules in a category is the obvious generalization of ordinary modules in the category of abelian groups and we expect most readers to feel comfortable with it.

Instead of an Alpine hypothyroid proof, we present in Section~\ref{se:monadicity} a more urbane approach, which also leads to nice generalizations. For instance, Theorems~\ref{thm:paradigm} and~\ref{thm:para-der} give us the very same statement for arbitrary subgroups $H\leq G$, of possibly infinite index, at the cost of replacing the ring object~$A$ in $\cat C(G)$ by a ``ring functor" $\bfA\,:\cat C(G)\to \cat C(G)$, better known as a \emph{monad}. A similar theorem holds for a so-called ``cyclic shifted subgroup" of an elementary abelian group; see Theorem~\ref{thm:scs}.

If the reader prefers category-theory language, these theorems actually establish \emph{monadicity} of various restriction-coinduction adjunctions. See Remark~\ref{rem:Beck}.

\smallbreak

Beyond its counter-intuitive simplicity, Theorem~\ref{thm:main} is particularly remarkable in cases~\eqref{it:2} and~\eqref{it:3}, for derived and stable categories, because we really mean here ``modules in the homotopy category" and not ``homotopy category of modules"! In other words, these triangulated categories $\cat C(H)$ can be obtained via a purely triangulated-categorical construction applied to~$\cat C(G)$; see~\cite{Balmer11}. To put things in perspective, let us draw an analogy with algebraic geometry.

For a noetherian scheme~$X$ (say, a variety), the functor on derived categories $\Der(X)\to \Der(U)$ induced by restriction to an open subscheme $U\subset X$ is a categorical localization. However, when $\cat C(G)$ is the derived or the stable category of a finite group~$G$, no localization of~$\cat C(G)$ comes anywhere close to~$\cat C(H)$, in general. The point we make here is that this passage from $G$ to $H$ is obtained via separable monads. (Note that localizations are very special monads.) In algebraic geometry, allowing separable monads instead of just localizations is basically the same thing as allowing \'etale covers instead of just Zariski covers. Hence, transposing \'etale extensions to representation theory is much richer than transposing only localizations. In fact, it is an open question whether there is more ``\'etale topology" in modular representation theory beyond restriction to subgroups. See Remark~\ref{rem:Conj}.

This being said, the main motivation for Theorem~\ref{thm:main} is the change of paradigm that it suggests. Indeed, since $\cat C(H)$ turns out to be the category of $A$-modules in $\cat C(G)$, the problem of \emph{extending} representations from~$H$ to~$G$ now becomes a \emph{descent} problem in $\cat C(G)$ with respect to the ring~$A=A^G_H$. In algebraic geometry, descent has been systematically studied by Grothendieck and the Diadochi and applies to many frameworks, including monads; see Mesablishvili~\cite{Mesablishvili06a}. Descent is pretty well-behaved for triangulated categories too, as explained in~\cite{Balmer12}, which allows us to discuss descent in derived and stable categories. The critical condition for descent to hold is that $A^G_H$ should be faithful, which amounts to the index $\Index{G}{H}$ being invertible in~$\kk$. See Remark~\ref{rem:desc}.

One could then try to express descent with respect to~$A$ by means of $A$-modules equipped with gluing isomorphisms in $A\potimes2$-modules satisfying cocycle conditions in $A\potimes 3$-modules. We explain in the same Remark~\ref{rem:desc} that this strategy collapses in an embroglio of Mackey formulas and an overdose of non-natural choices. To master these technicalities, it is convenient to replace subgroups of~$G$ by $G$-orbits. This leads us in Part~\ref{part:B} to a Grothendieck topology and to stacks, as we now explain.

\smallbreak

For simplicity, we assume for the rest of this introduction that $G$ is finite and that $\kk$ is a field of characteristic~$p$. Transposing~\ref{not:123} to $G$-sets, we get\,:
\begin{Not}
\label{not:123'}%
For every finite $G$-set~$X$, we write $\cat D(X)$ for the following category
\begin{enumerate}[(1')]
\item\label{it:1'}%
\quad $\cat D(X)= \Rep(X)$ \kern3.1em the category of representations of $X$, in case~\eqref{it:1},
\item\label{it:2'}%
\quad $\cat D(X)= \Der(\Rep(X))$ \kern 1.6em its derived category, in case~\eqref{it:2},
\item\label{it:3'}%
\quad $\cat D(X)= \StRep(X)$ \quad its stable category, in case~\eqref{it:3}.
\end{enumerate}
The category $\Rep(X)=(\kk\VVect)^{G\ltimes X}$ is defined via the action groupoid $G\ltimes X$.
\end{Not}

This standard material is recalled in the short Section~\ref{se:rep} for the reader's convenience. Among these categories $\cat D(X)$, we find our original categories $\cat C(H)$ for~$H\leq G$ as in~\ref{not:123}, simply by considering orbits. Indeed: $\cat C(H)\cong \cat D(G/H)$. This idea roots back to Dress~\cite{Dress73}. Since $G$-maps from $G/H_1$ to~$G/H_2$ are given by elements of $G$ which normalize $H_1$ into~$H_2$, these functors $\cat D(-)$ allow us to treat simultaneously conjugation and restriction to subgroups. Hence $\cat D(-)$ might be apprehended as a categorification of ordinary Mackey functors, see Webb~\cite{Webb00}. In other words, $\cat D(-)$ is a presheaf of categories on the category of finite $G$-sets. Descent will tell us something more, namely that $\cat D(-)$ is in fact a \emph{sheaf} in the appropriate sense.

As in algebraic geometry, we use the notion of \emph{stack} to formalize the above heuristical ``sheaf of categories"; see Section~\ref{se:descent}. The central Theorem~\ref{thm:stack} tells us that these presheaves of representations $\cat D(-)$ define stacks on the category of finite $G$-sets with respect to a suitable Grothendieck topology, called the \emph{sipp topology}. By the above discussion, we expect a subgroup $H\leq G$ to ``cover'' $G$ if its index $\Index{G}{H}$ is prime to~$p=\chara(\kk)$. Translated in terms of the associated $G$-map on orbits $G/H\onto G/G$, we want \underbar{s}tabilizers to have \underbar{i}ndex \underbar{p}rime to~\underbar{$p$}, hence the name \emph{sipp} topology. For clarity, we describe this topology on $\Gsets$ in Section~\ref{se:top}, at the start of Part~\ref{part:B}, before even speaking of $G$-set representations. Alternatively, we could restrict the sipp topology to the orbit category $\Or(G)$ and the theory would go through. It is more convenient to work with the whole category of $G$-sets because its has pull-backs, whereas $\Or(G)$ does not, but this choice is mostly cosmetic.

\smallbreak

Turning to applications in Part~\ref{part:C}, we want to use descent to extend modular representations from a subgroup $H$ to the group~$G$ when $\Index GH$ is prime to~$p$. In other words, we want to apply the methods of Part~\ref{part:B} to the stable categories of~\eqref{it:3}~\&~(\ref{it:3'}'). Once we understand $U:=G/H$ as a sipp-cover of $X:=G/G$, the descent property  involves gluing isomorphisms on the ``intersection" $U\times_XU$ and cocyle conditions on the ``double intersection" $U\times_XU\times_XU$. If we try to translate this in terms of subgroups, we bump into Mackey formulas again. So where was the gain\,? The answer is a standard (Grothendieckian) trick\,: First, accept \emph{all} choices and then deal with the excess of information. The first step of this strategy is best implemented with representations of $G$-sets and leads us to the hybrid Theorem~\ref{thm:tame} which still involves $\StRep(-)$ but is free of any Mackey-formulaic choices. The next step, in Section~\ref{se:extend}, is to restore usual stable categories $\kk ?\SStab$ of subgroups instead of all the $\StRep(G/?)$ in sight. This turns Theorem~\ref{thm:tame} into the following plug-and-play result (Theorem~\ref{thm:stack0}), which can be used without any knowledge of stacks and Grothendieck topology\,:
\begin{Thm}
\label{thm:pub}%
Let $H\leq G$ be a subgroup of index prime to~$p$. Let $W$ be a $\kk H$-module. For every $g\in G$, let $\sigma_g:W\dto_{H[g]}\,\isotoo \con{g\,}W\dto_{H[g]}$ be an isomorphism in the stable category $\kk H[g]\SStab$, where $H[g]$ stands for $H^g\cap H$ and where $\con{g\,}W\dto_{H[g]}$ is $g$-twisted restriction (see Notation~\ref{not:gRes} if needed), with the following hypotheses\,:
\begen[(I)]
\item
If $h\in H$ (so $H[h]=H$), assume that the given isomorphism $\sigma_h$ and the canonical isomorphism $h\cdot:W\otoo{\cong} \con{h\,}W$, $w\mapsto h\,w$, are equal in $\kk H\SStab$.
\smallbreak
\item
For every $g_1,g_2\in G$, consider the subgroup ${H[g_2,g_1]}:=H^{g_{2}g_{1}}\cap H^{g_1}\cap H$ and assume that the following diagram commutes in $\kk {H[g_2,g_1]}\SStab$\,:
$$
\xymatrix@C=2em@R=1.5em{
& W\dto_{H[g_2,g_1]}
 \ar[ld]_-{\Displ\,\sigma_{g_1}}
 \ar[rd]^-{\Displ\,\sigma_{g_2g_1}}
\\
\con{g_1} W\dto_{H[g_2,g_1]}\
 \ar[rr]^-{\Displ\,\sigma_{g_2}}
&& \ \con{g_2g_1} W\dto_{H[g_2,g_1]}\,. \kern-3em
}
$$
\ened
Then $W$ extends to~$G$, \ie there is a $\kk G$-module $V$ and an isomorphism $f: V\dto_H \isoto W$ in $\kk H\SStab$ such that for every $g\in G$ the following commutes in $\kk H[g]\SStab$\,:
$$
\xymatrix@C=4em@R=1.8em{
V \dto_{H[g]} \ar[r]^-{f}_-{\simeq} \ar[d]_-{\Displ g\cdot}^-{\cong}
& W \dto_{H[g]} \ar[d]^-{\Displ\sigma_g}_(.45){\simeq}
\\
\con{g\,}V \dto_{H[g]} \ar[r]^-{f}_-{\simeq}
& \con{g\,}W \dto_{H[g]}\,.\kern-1em
}
$$
Moreover, the pair $(V,f)$ is unique up to unique isomorphism, in the obvious sense.
\end{Thm}

To measure the importance of this application, note that it constitutes a substantial generalization of the main result of~\cite{Balmer12app}, where we treated the special case of the trivial representation~$W=\kk$, in order to compute the kernel of the restriction homomorphism $T(G)\to T(P)$, where $P$ is a Sylow $p$-subgroup of~$G$ and where $T(G)$ is the group of endotrivial $\kk G$-modules. The general Theorem~\ref{thm:pub} above gives a criterion to extend arbitrary representations~$W$ and is therefore important beyond endotrivial ones. Interestingly, even for endotrivial representations, it also allows us to improve on~\cite{Balmer12app} and describe the \emph{image} of $T(G)\to T(P)$. The non-specialist will find in~\cite{Balmer12app,Bouc06,CarlsonThevenaz04,CarlsonThevenaz05} further references on the central role played by endotrivial modules in modular representation theory.

Carlson-Th\'evenaz~\cite{CarlsonThevenaz04,CarlsonThevenaz05} classified the groups $T(P)$ for all $p$-groups~$P$. For arbitrary finite groups~$G$, the invariant $T(G)$ is not given by a simple formula and no classification is expected to exist in general. So the problem is to describe as explicitly as possible the kernel and the image of the restriction homomorphism $T(G)\to T(P)$, for $P\leq G$ a Sylow $p$-subgroup, knowing that the actual computation for every given group will remain difficult. Note that the group $T(G)$ is nothing but the Picard group of $\otimes$-invertible objects in the stable category\,: $T(G)=\Pic(\kk G\sstab)$. Contrary to its algebro-geometric counterpart, this representation-theoretic Picard group $T(G)$ is not an $\rmH^1(-,\Gm)$ in any known way. However, although neither $T(G)$ nor $T(P)$ are cohomology groups, we prove here that $\Ker(T(G)\to T(P))$ and $\Img(T(G)\to T(P))$ are related to the first and second \Cech\ cohomology groups of the sipp sheaf of units~$\Gm$, which is just the constant sheaf associated to the abelian group of units~$\kk^\times$. Indeed, if we consider the  sipp-cover~$\cat U:=\{G/P\to G/G\}$, Theorem~\ref{thm:TGP} gives a canonical isomorphism
$$
\Ker\big(T(G)\to T(P)\big)\cong \cH^1(\cat U,\Gm)\,.
$$
This formula recovers and conceptualizes the main result of~\cite{Balmer12app}, which was more down-to-earth. On the other hand, the result about the image is new and reads
$$
\Img\big(T(G)\to T(P)\big)\cong \Ker\big(\cH^0(\cat U,\Pic)\oto{z} \cH^2(\cat U,\Gm)\big)
$$
for an explicit group homomorphism $z:\cH^0(\cat U,\Pic)\to \cH^2(\cat U,\Gm)$, see Theorem~\ref{thm:H2}. These \Cech\ cohomology groups give an ideal solution to the problem of determining $T(G)$ for \emph{all} groups~$G$, because they basically only involve the action of $G$ on its $p$-subgroups (see Definition~\ref{def:Cech}). In particular, they do not involve any representations, nor any stable categories. Although most probably possible, the ``numerical" determination of these groups for specific groups~$G$ is left to more computer-savvy people than the author.

%------------------------------------------------------------------------------
%------------------------------------------------------------------------------
\goodbreak
\part{Restriction via separable extension-of-scalars}
\label{part:A}%
\bigbreak
%------------------------------------------------------------------------------

%------------------------------------------------------------------------------
%\medbreak
\section{Categories of modules and monadicity}
\label{se:monadicity}%
\medbreak
%------------------------------------------------------------------------------

%We briefly recall some basics of category theory, mostly about monads.

%
\begin{Rem}
\label{rem:ic}%
An additive category $\cat C$ is \emph{idempotent-complete} (or karoubian) if every idempotent morphism $e=e^2:X\to X$ in $\cat C$ yields a decomposition $X=\im(e)\oplus \ker(e)$. Any additive category can be idempotent-completed $\cat C\hook\cat C^\natural$ by an elegant well-known construction due to Karoubi. An additive functor $F:\cat C\to \cat D$ is an \emph{equivalence up to direct summands} if the induced functor $F^\natural:\cat C^\natural\to\cat D^\natural$ is an equivalence. This is the same as saying that $F:\cat C\to \cat D$ is fully-faithful and that every object in~$\cat D$ is a direct summand of the image by $F$ of some object of~$\cat C$.
\end{Rem}

We now recall the concept of monad on a category~$\cat C$; see~\cite{MacLane98}. In short, a monad on~$\cat C$ is a monoid in the category of endofunctors.

\begin{Def}
\label{def:monad}%
A \emph{monad} $(\bfA,\mu,\eta)$ on~$\cat C$ is an
endo-functor $\bfA:\cat C\to \cat C$ with a natural transformation $\mu:\bfA^2\to \bfA$, called
the \emph{multiplication}, such that $\mu\circ (\bfA\mu)=\mu\circ (\mu
\bfA):\, \bfA^3\to \bfA$ (associativity) and with a natural transformation
$\eta:\Idcat{C}\to \bfA$, called the \emph{two-sided unit}, such that
$\mu\circ (\bfA\eta)=\mu\circ(\eta \bfA)=\id_\bfA:\, \bfA\to \bfA$.

An \emph{$\bfA$-module in~$\cat C$} is a pair $(X,\varrho)$ where $X$ is an object of~$\cat C$ and $\varrho:\bfA(X)\to X$ is a morphism in~$\cat C$, called the \emph{$\bfA$-action}, such that $\varrho\circ (\bfA\varrho)=\varrho\circ\mu_X:\, \bfA^2(X)\to X$ and $\varrho\circ\eta_X=\id_X:\,X\to X$. These replace the usual $a\cdot (b\cdot x)=(ab)\cdot x$ and $1\cdot x=x$ for ordinary modules. Morphisms of $\bfA$-modules $f:(X,\varrho)\to (X',\varrho')$ are
morphisms $f:X\to X'$ in~$\cat C$ such that $\varrho'\circ \bfA(f)=f\circ\varrho:\bfA(X)\to X'$ ($\bfA$-linearity), replacing $a\cdot f(x)=f(a\cdot x)$. We denote by $\bfA\MModcat{C}$ the
category of $\bfA$-modules in~$\cat C$ and we have the so-called
\emph{Eilenberg-Moore~\cite{EilenbergMoore65} adjunction}
$$
F_\bfA\,:\,\cat {C}\adj \bfA\MModcat{C}\,:\,U_\bfA
$$
where the left adjoint is \emph{extension-of-scalars},
$F_\bfA(Y)=\big(\bfA(Y),\mu_Y\big)$ and $F_\bfA(f)=\bfA(f)$, and the right
adjoint is the forgetful functor $U_\bfA(X,\varrho)=X$ and $U_\bfA(f)=f$. The $\bfA$-module $F_{\bfA}(Y)$ is also called the \emph{free $\bfA$-module} over~$Y$.

Denote by $\bfA\FFreecat{C}$ the full subcategory of $\bfA\MModcat{C}$ on free $\bfA$-modules. Equivalently but more accurately, $\bfA\FFreecat{C}$ is taken to have the same objects as~$\cat C$ and morphisms of associated free modules. The Eilenberg-Moore
adjunction restricts to the so-called \emph{Kleisli~\cite{Kleisli65} adjunction}
$F_\bfA\,:\,\cat{C}\adj \bfA\FFreecat{C}\,:\,U_\bfA$. See~\eqref{eq:KE}.

Dually for comonads, comodules, etc. See~\cite[App.\,A]{Balmer12} if needed.
\end{Def}

\begin{Exa}
\label{exa:ring-obj}%
If the category $\cat C $ is monoidal with tensor $\otimes:\cat C\times \cat C\to \cat C$ and unit~$\unit$, a \emph{ring object} $(A,\mu,\eta)$ in~$\cat C$ is an object $A\in\cat C$ with morphisms $\mu:A\otimes A\to A$ and $\eta:\unit\to A$ such that $A\otimes-$ becomes a monad on~$\cat C$. Then $A$-modules are pairs $(X,\varrho)$ where $X$ is an object of~$\cat C$ and $\varrho:A\otimes X\to X$ is a morphism in~$\cat C$, which satisfies the above relations. When $\cat C=\bbZ\MMod$, this yields ordinary rings and modules.
\end{Exa}

\begin{Rem}
\label{rem:A-ic}%
If $\cat C$ is additive (resp.\ idempotent-complete) and the monad $\bfA$ is an additive functor, then $\bfA\MModcat{C}$ is additive (resp.\ idempotent-complete) as well.
\end{Rem}

\begin{Rem}
\label{rem:KE}%
Every adjunction $L\,:\,\cat C\adj \cat D\,:\,R$ with unit
$\eta:\Idcat{C}\to RL$ and counit $\eps:LR\to \Idcat{D}$ gives rise
to a monad on~$\cat C$, defined by $\bfA=RL$, $\eta=\eta$ and $\mu=R\eps
L$. One says that the adjunction $(L,R,\eta,\eps)$ \emph{realizes}
the monad~$(\bfA,\mu,\eta)$. The Kleisli and Eilenberg-Moore adjunctions of
Definition~\ref{def:monad} are respectively initial and final among adjunctions realizing~$\bfA$, see~\cite[Chap.\,VI]{MacLane98}. That is, if an adjunction $L\,:\,\cat C\adj \cat D\,:\,R$ realizes a monad~$\bfA$ then there exist unique functors $K:\bfA\FFreecat{C}\to \cat D$ and $E:\cat D\to \bfA\MModcat{C}$ as in the following diagram\,:
\begin{equation}
\label{eq:KE}%
\vcenter{\xymatrix@C=6em@R=2.5em{
& \cat{C}\vphantom{_Y} \ar@<-.2em>[ld]_-{F_\bfA} \ar@<-.2em>[d]_(.55){L} \ar@<-.2em>[rd]_(.55){F_\bfA}
 \ar@(lu,ru)[]^(.22){\bfA}
\\
\bfA\FFreecat{C} \ar[r]_-{K} \ar@<-.2em>[ru]_-{U_\bfA}
 \ar@{_(->}@/_1.5em/[rr]|{\textrm{\ fully faithful\ }}
& \cat{D} \ar@<-.2em>[u]_(.45){R}
\ar[r]_-{E}
&
\bfA\MModcat{C} \ar@<-.2em>[lu]_-{U_\bfA}
}}
\end{equation}
such that $K\circ F_\bfA=L$, $R\circ K=U_\bfA$, $E\circ L=F_\bfA$ and
$U_\bfA\circ E=R$. Explicitly, $K$ is defined by $K(F_{\bfA}(Y))=L(Y)$ on objects and by the isomorphism $\Hom_{\bfA\FFreecat{C}}(F_\bfA(Y),F_\bfA(Y'))\cong\Homcat{C}(Y,\bfA(Y'))\cong\Homcat{D}(L(Y),L(Y'))$ on morphisms. In particular, $K$ is always fully-faithful. On the other hand, we have
\begin{equation}
\label{eq:E}%
E(Z)=\big(R(Z),R(\eps_Z)\big)\qquadtext{and}E(f)=R(f)
\end{equation}
for every object $Z$ and every morphism $f$ in~$\cat D$. Note also that
$E\circ K$ is the fully faithful inclusion of $\bfA\FFreecat{C}$
into~$\bfA\MModcat{C}$.

An adjunction $L\,:\,\cat C\adj \cat D\,:\,R$ is called \emph{monadic} when the Eilenberg-Moore functor $E:\cat D\to (RL)\MMod_{\cat C}$ of~\eqref{eq:KE} is an equivalence of categories. Dually, it is \emph{comonadic} if the Eilenberg-Moore functor $\cat C\to (LR)\CComod_{\cat D}$ is an equivalence.
\end{Rem}

\begin{Rem}
\label{rem:Beck}%
There is a famous Monadicity Theorem of Beck~\cite{Beck03} which we do not use here for various reasons. First, in our case the proof is much simpler. Lemma~\ref{lem:monadicity} below has all the monadicity we need and is considerably faster to state, prove and use than Beck's result. Most important, we need to extend our results to triangulated categories, where we should avoid colimits, like coequalizers. Yet Lemma~\ref{lem:monadicity} easily extends to triangulated categories and also gives \emph{separability} of~$\bfA=RL$, which is crucial to put a triangulation on $\bfA\MModcat{C}$, by~\cite[\S\,4]{Balmer11}.
\end{Rem}

\begin{Def}
\label{def:sep-monad}%
A monad $\bfA:\cat C\to \cat C$ is called a \emph{separable monad} if $\mu:\bfA^2\to \bfA$ admits an ``$\bfA$-bilinear section", \ie a natural transformation $\sigma:\bfA\to \bfA^2$ such that $\mu\circ\sigma=\id_\bfA$ and
$(\bfA\mu)\circ(\sigma \bfA)=\sigma\circ\mu=(\mu
\bfA)\circ(\bfA\sigma)\,:\,\bfA^2\to \bfA^2$.
\end{Def}

\begin{Lem}
\label{lem:monadicity}%
Let $L\,:\,\cat C\adj \cat D\,:\,R$ be an adjunction of functors between additive categories. Suppose that the counit $\eps:L\, R\to\Idcat{D}$ has a natural section, \ie there is a natural transformation $\xi:\Idcat{D}\to L\, R$ such that
$\eps\circ\xi=\id$. Then\,:
\begen
\item The induced monad $\bfA=R\,L$ on~$\cat C$ is separable.
\smallbreak
\item The Kleisli and Eilenberg-Moore comparison functors
$K:\bfA\FFreecat{C}\to \cat D$ and $E:\cat{D}\to \bfA\MModcat{C}$
of~\eqref{eq:KE} are equivalences up to direct summands.
\smallbreak
\item If we assume moreover that $\cat C$ and $\cat D$ are
idempotent-complete then $E:\cat{D}\isoto \bfA\MModcat{C}$ is an
equivalence, \ie the adjunction is monadic.
\ened
\end{Lem}

\begin{proof}
All this is standard but we sketch the proof for the reader's convenience. First verify that $\sigma:=R\xi L:\bfA=RL\too RLRL=\bfA^2$ is a section of~$\mu=R\eps L$ and that $\sigma$ is ``$\bfA$-bilinear" by naturality of~$\eps$ and~$\xi$. This gives~(a). When $\bfA$ is separable with section~$\sigma$ of~$\mu$, one can verify that for every $\bfA$-module $(X,\varrho)$ the $\bfA$-linear morphism $\varrho:F_\bfA(X)\to X$, from the free $\bfA$-module $F_\bfA(X)$ to~$X$, admits an $\bfA$-linear retraction given by $\bfA(\varrho)\,\sigma\,\eta_X:X\to F_\bfA(X)$; see~\cite[Prop.\,6.3]{BruguieresVirelizier07}. Consequently, the fully faithful inclusion $\bfA\FFreecat{C}\hook \bfA\MModcat{C}$ is an equivalence up to direct summands (Remark~\ref{rem:ic}). Since this inclusion coincides with $E\circ K$, it suffices to prove~(b) for $K$ to get both. Recall that $K:\bfA\FFreecat{C}\to \cat {D}$ is fully faithful already. Now, for every~$D\in\cat D$, we have by hypothesis $\eps_D\circ\xi_D=\id_D:D\to LR(D)\to D$ which shows that $D$ is a direct summand of~$LR(D)=K(F_\bfA(R(D)))$. Hence $K$ is essentially surjective up to direct summands. This finishes~(b). Finally~(c) is a direct consequence of the equivalence $\ic{E}:\iccat{D}\isoto\ic{\bfA\MModcat{C}}$ of~(b), since both $\cat D$ and $\bfA\MModcat{C}$ are idempotent-complete (Remark~\ref{rem:A-ic}).
\end{proof}

An ancestor of our Theorem~\ref{thm:main} is the following application of Lemma~\ref{lem:monadicity}\,:

\begin{Thm}
\label{thm:ancestor}%
Let $\ell:D\to C$ be a homomorphism of (ordinary) rings, which admits a homomorphism of $(D,D)$-bimodules $m:C\to D$ with $m\circ \ell=\id_D$. Consider the usual restriction-coinduction adjunction on categories of left modules\,:
\begin{equation}
\label{eq:adj}%
\vcenter{\xymatrix{
\cat{C}:=C\MMod \ar@<-.5em>[d]_-{\Displ\Res_\ell\cong C\otimes_C-\ }
\\
\cat D:=D\MMod  \ar@<-.5em>[u]_-{\Displ\ \CoInd_{\ell}=\Hom_{D}(C,-)}
}}
\end{equation}
\begen[(a)]
\item The counit $\eps:\Res_\ell\CoInd_\ell\to\Idcat{D}$ admits a natural section.
\smallbreak
\item
Adjunction~\eqref{eq:adj} is monadic (Remark~\ref{rem:KE}), \ie for the monad $\bfA=\CoInd_\ell\Res_\ell$ on~$\cat C$, the functor $E:\cat D\too \bfA\MMod_{\cat C}$ of~\eqref{eq:KE} is an equivalence turning $\Res_\ell$ into the extension-of-scalars functor $F_{\bfA}$ and $\CoInd_\ell$ into the forgetful~$U_\bfA$.
\ened
\end{Thm}

\begin{proof}
By Lemma~\ref{lem:monadicity}, it suffices to show~(a). Recall that for a $(D,D)$-bimodule~$B$, the abelian groups $\Hom_{D}(B,?)$ are \emph{left} $D$-modules via right action on~$B$, that is, $(d\cdot f)(b):=f(b\,d)$. Therefore, a homomorphism $k:B\to B'$ of $(D,D)$-bimodules induces a natural transformation $k^*:\Hom_D(B',?)\to \Hom_D(B,?)$, $f\mapsto f\circ k$, of endofunctors on~$\cat D$. Apply this to $k=\ell$ and $k=m$. Under the identification $\Id_{\cat C}\cong\Hom_D(D,?)$, the counit $\eps:\Res_\ell(\CoInd_\ell(?))=\Hom_D(C,?)\to \Hom_D(D,?)\cong\Id_{\cat D}(?)$ of adjunction~\eqref{eq:adj} is simply given by $\eps=\ell^*$. Hence a section $\xi$ of~$\eps$ is given by $\xi=m^*$ since $\eps\,\xi=\ell^*m^*=(m\ell)^*=\id$.
\end{proof}

%------------------------------------------------------------------------------
\medbreak
\section{Restriction of group representations}
\label{se:ResGH}%
\medbreak
%------------------------------------------------------------------------------

In this section, $H\leq G$ is a subgroup, without any finiteness assumption at first.

\begin{Rem}
Consider adjunction~\eqref{eq:adj} for $\ell:\kk H\hook \kk G$ the inclusion\,:
\begin{equation}
\label{eq:res-coind}%
\vcenter{\xymatrix{
\kk G\MMod \ar@<-.5em>[d]_{\Displ\Res^G_H\ } \\
\kk H\MMod\,. \ar@<-.2em>[u]_-{\Displ\ \CoInd_H^G=\Hom_{\kk H}(\kk G,-)}
}}
\end{equation}
We abbreviate $\Res=\Res^G_H$ and $\CoInd=\CoInd_H^G$ when no confusion can occur.
Explicitly, the unit $\eta_V:V\to
\CoInd(\Res(V))=\Hom_{\kk H}(\kk G,V)$ for $V$ in $\kk G\MMod$ and the
counit $\eps_W:\Hom_{\kk H}(\kk G,W)=\Res(\CoInd(W))\to W$ for
$W$ in $\kk H\MMod$ are given for every $v\in V$, $x\in \kk G$ and $f\in \Hom_{\kk H}(\kk G,W)$ by the formulas\,:
\begin{equation}
\label{eq:eta-eps}%
(\eta_V(v))(x)=x\cdot v\qquadtext{and}\eps_W(f)=f(1)\,.
\end{equation}
\end{Rem}

\begin{Rem}
\label{rem:A}%
Applying Remark~\ref{rem:KE} to the restriction-coinduction adjunction~\eqref{eq:res-coind}, we obtain a monad $\bfA=\bfA^G_H$ on $\kk G\MMod$ given by $\bfA=\CoInd_H^G\circ\Res^G_H$, that is, $\bfA(V)=\Hom_{\kk H}(\kk G,V)$ for every $V$ in $\kk G\MMod$. The left $\kk G$-action on $\Hom_{\kk H}(\kk G,V)$ is via right action on~$\kk G$. The unit $\eta:\Id\to \bfA$ is exactly the one of~\eqref{eq:eta-eps}. Multiplication $\mu:\bfA^2\to \bfA$ is $\mu=\CoInd\eps\Res$, where $\eps:\Res\CoInd\to \Id$ is the counit given in~\eqref{eq:eta-eps}. So, for every $V$ in $\kk G\MMod$, we explicitly give
$$
\mu_V\,:\ \bfA^2(V)=\Hom_{\kk H}(\kk G,\Hom_{\kk H}(\kk G,V)) \too
\Hom_{\kk H}(\kk G,V)=\bfA(V)
$$
by $\big(\mu_V(f)\big)(x)=(f(x))(1)$, for every $f\in \bfA^2(V)$ and every~$x\in \kk G$.
\end{Rem}

We temporarily denote by $\cat C(G)=\kk G\MMod$ the category of left $\kk G$-modules.
\begin{Thm}
\label{thm:paradigm}%
Let $H\leq G$ be an arbitrary subgroup. Let $\bfA=\bfA^G_H$ be the monad on $\cat C(G)$ induced by the restriction-coinduction adjunction. By Eilenberg-Moore, see~\eqref{eq:KE}, we have the following diagram in which $E\circ\Res=F_\bfA$ and $U_\bfA\circ E=\CoInd$\,:
$$
\xymatrix@C=3em@R=2em{
& \cat C(G)\vphantom{_Y}
 \ar@<-.2em>[ld]_(.55){\Res} \ar@<-.2em>[rd]_(.45){F_\bfA}
\\
\cat C(H)\ar@<-.2em>[ru]_(.55){\CoInd}
\ar[rr]_-{E}
&&
\bfA\MModcat{C(G)} \ar@<-.2em>[lu]_-{U_\bfA}
}
$$
Then $E$ is an equivalence. In words, the category $\cat C(H)$ is equivalent to $\bfA$-modules in $\cat C(G)$ in such a way that restriction coincides with extension-of-scalars with respect to $\bfA$ and coinduction coincides with the functor which forgets $\bfA$-actions.
\end{Thm}

\begin{proof}
To apply Theorem~\ref{thm:ancestor}, it suffices to show that $\ell:\kk H\hook \kk G$ has a retraction $m:\kk G\to \kk H$ as $(\kk H,\kk H)$-bimodule. For every $g\in G$, define $m(g)=g$ if $g\in H$ and $m(g)=0$ if $g\notin H$ and extend it $\kk$-linearly to get the wanted $m:\kk G\to \kk H$.
\end{proof}

\begin{Rem}
\label{rem:EGH}%
By~\eqref{eq:E}, for every $\kk H$-module $W$, the $\bfA$-module $E(W)=(V, \varrho)$
is given by $V:=\CoInd_H^G(W)=\Hom_{\kk H}(\kk G,W)$ with $\bfA$-action $\varrho$ in~$\cat C(G)$
$$
\xymatrix@C=8em{\bfA(V) = \Hom_{\kk H}(\kk G,\Hom_{\kk H}(\kk G,W))\ar[r]^-{\Displ\varrho:=\CoInd(\eps_W)}
& \Hom_{\kk H}(\kk G,W) = V}
$$
given by $\big(\varrho(f)\big)(x)=\big(f(x)\big)(1)$, for every $f\in \bfA(V)$ and every $x\in \kk G$.
\end{Rem}

\begin{center}*\ *\ *\end{center}
\goodbreak

For the rest of the section, we further assume that the index $\Index{G}{H}$ is finite.

\begin{Rem}
We can now replace coinduction by induction and get the adjunction
\begin{equation}
\label{eq:res-ind}%
\vcenter{\xymatrix{
\cat C(G)=\kk G\MMod \ar@<-.5em>[d]_{\Displ\Res^G_H} \\
\cat C(H)=\kk H\MMod\,. \ar@<-.2em>[u]_-{\Displ\Ind_H^G=\kk G\otimes_{\kk H}-}
}}
\end{equation}
The unit $\eta'_V:V\to \Ind(\Res(V))=\kk G\otimes_{\kk H}V$ for $V$ in $\kk G\MMod$ and the counit $\eps'_W:\kk G\otimes_{\kk H}W=\Res(\Ind(W))\to W$ for $W$ in $\kk H\MMod$ are given by the formulas\,:
\begin{equation}
\label{eq:eta-eps'}%
\eta'_V(v)=\sum_{[x]_{H}\,\in\, G/H}x\otimes x\inv v\qquadtext{and}
\eps'_W(g\otimes w)=
\left\{\begin{array}{cc}g\cdot w &\text{if }g\in H
\\[.5em]
0&\text{if }g\notin H
\end{array}\right.
\end{equation}
for every $v\in V$, $g\in G$ and $w\in W$. (We write $[x]_H$ for $xH$ everywhere.) We denote by $\bfA\!'=\Ind\Res$ the monad on $\cat C(G)$ associated to this adjunction (Remark~\ref{rem:KE}).
\end{Rem}

\begin{Cor}
\label{cor:para-fin}%
With the above notation, adjunction~\eqref{eq:res-ind} is monadic, \ie the associated Eilenberg-Moore functor $E':\cat C(H)\to \bfA\!'\MMod_{\cat C(G)}$ is an equivalence.
\end{Cor}

\begin{proof}
Since $\Index{G}{H}<\infty$, we have the well-known isomorphism $\Ind\isoto\CoInd$; see~\cite[\S\,I.3.3]{Benson98}. We already know that the $\Res/\CoInd$ adjunction is monadic by Theorem~\ref{thm:paradigm}, then so is the isomorphic $\Res/\Ind$ adjunction. Alternatively, one can apply Lemma~\ref{lem:monadicity} and prove directly that the counit $\eps'$ has a natural section $\xi':\Id_{\cat C(H)}\to \Res\Ind$ given for every $\kk H$-module $W$ by $\xi'_W(w)=1\otimes w$.
\end{proof}

\begin{Rem}
\label{rem:Frob}%
Recall that $\kk G\MMod$ is symmetric monoidal via $V_1\otimes V_2=V_1\otimes_\kk V_2$ with diagonal $G$-action. Also recall the natural isomorphism of $\kk G$-modules
$$
\vartheta_V\,:\ \bfA\!'(V)=\kk G\otimes_{\kk H}V\isotoo \kk(G/H)\otimes_\kk V
$$
for every $V$ in $\kk G\MMod$, where the $G$-action on $\bfA\!'(V)$ is on~$\kk G$ only, whereas the $G$-action on $\kk(G/H)\otimes_\kk V$ is the diagonal one. This isomorphism $\vartheta_V$ is given by
\begin{equation}
\label{eq:theta}%
\vartheta_V(g\otimes v)=[g]_H\otimes gv
\end{equation}
for every $g\in G$ and $v\in V$. Its inverse is given for every $\gamma\in G/H$ and $v\in V$ by
\begin{equation}
\label{eq:theta-inv}%
\kern2em\vartheta_V\inv(\gamma\otimes v )=g\otimes g\inv v
\quad\text{ for any choice of }g\in \gamma
\end{equation}
Consequently, the monad $\bfA\!'=\Ind\Res$ is isomorphic, as a functor, to $A\otimes-$ for $A=\kk(G/H)$ and the latter will inherit a structure of ring object. However, this does not imply that $\vartheta$ is an isomorphism of monads\,! So, let us be precise.
\end{Rem}

\begin{Def}
\label{def:AH}%
Define the $\kk G$-module $A=A^G_H$ to be the free $\kk$-module $\kk(G/H)$
with basis $G/H$ with obvious left $G$-action on the $\kk$-basis\,:
$g\cdot[x]_H=[gx]_H$. As in the Introduction, we define a $\kk G$-linear
morphism
$$
\mu: A^G_H\otimes_\kk A^G_H\too A^G_H
$$
by $\mu(\gamma\otimes\gamma)=\gamma$ and
$\mu(\gamma\otimes\gamma')=0$ if $\gamma\neq\gamma'$ in~$G/H$.
Since $H\leq G$ has finite index, we define the $\kk G$-linear map $\eta:\unit\to A^G_H$, \ie
$\eta:\kk\to \kk(G/H)$, by $1\mapsto\sum_{\gamma\in G/H}\gamma$.
\end{Def}

\begin{Rem}
Ignoring $G$-actions, this ring would be silly (just $\Index{G}{H}$ copies of~$\kk$) and its category of plain modules would consist of the direct sum of $\Index{G}{H}$ copies of the category of $\kk$-modules. Again, it is important to consider the ring object $A$ in~$\cat C(G)$, that is, to keep track of the $G$-action on~$A$, and to consider the category $A\MMod_{\cat{C}(G)}$ of $A$-modules \emph{in the category~$\cat C(G)$}, as emphasized already.
\end{Rem}

\begin{Prop}
\label{prop:AH}%
Let $H\leq G$ be a finite-index subgroup. Then\,:
\begen
\item
The triple $(A,\mu,\eta)$ of Definition~\ref{def:AH} is a commutative ring object in the symmetric monoidal category~$\kk G\MMod$ (also finite-dimensional over~$\kk$).
\smallbreak
\item
The ring object $A$ is \emph{separable} (Definition~\ref{def:sep-monad}), \ie there exists a section $\sigma:A\to A\otimes A$ of $\mu$ satisfying $(A\otimes\mu)\circ(\sigma \otimes
A)=\sigma\circ\mu=(\mu\otimes
A)\circ(A\otimes\sigma)$.
\smallbreak
\item
The monad $A\otimes-$ on $\kk G\MMod$ is isomorphic to the monad
$\bfA\!'=\Ind\Res$ associated to the restriction-induction adjunction~\eqref{eq:res-ind}. The explicit natural isomorphism $\vartheta:\bfA\!'\isoto A\otimes-$ is given in Remark~\ref{rem:Frob} above.
\ened
\end{Prop}

\begin{proof}
(a)\,: Associativity, two-sided unit and commutativity are easy exercises. Part (b) will follow from~(c) and the
separability of~$\bfA\!'\simeq \bfA$ but we can also provide $\sigma:A\to A\otimes A$
explicitly as $\sigma(\gamma)=\gamma\otimes \gamma$ for every
$\gamma\in G/H$. For~(c), we need to show that $\vartheta:\bfA\!'\isoto A\otimes -$ respects multiplications and units. The latter means $\vartheta_V\circ \eta'_V=\eta\otimes 1_V$ for every $V\in \kk G\MMod$ and is easy to verify using~\eqref{eq:eta-eps'}, \eqref{eq:theta} and Definition~\ref{def:AH}. For
compatibility with multiplication, we need to check commutativity of
the following square for every $\kk G$-module~$V$\,:
\begin{equation}
\vcenter{\label{eq:ring-hom}%
\xymatrix@C=3em{
{\bfA\!'\,}^2(V)=\kk G\otimes_{\kk H}(\kk G\otimes_{\kk H}V)
 \ar[r]_-{\Displ\vartheta^{(2)}_V}
 \ar[d]_-{\Displ\mu'_V}
& \kk(G/H)\otimes \kk(G/H)\otimes V=A\potimes{2}\otimes V \ar[d]^-{\Displ\mu\otimes \id_V}\kern-1em
\\
{\bfA\!'}(V)=\kk G\otimes_{\kk H}V \ar[r]^-{\Displ\vartheta_V}
& \kk(G/H)\otimes_\kk V=A\otimes V
}}
\end{equation}
The above morphism $\vartheta^{(2)}_V:\bfA\!'\circ \bfA\!'(V)\to A\otimes A\otimes V$ is by definition $\vartheta$ applied twice in any order, that is, $\vartheta^{(2)}_V=(A\otimes\vartheta_V)\circ\vartheta_{{\bfA\!'}(V)}=\vartheta_{A\otimes V}\circ \bfA\!'(\vartheta_{V})$. (These coincide by naturality of~$\vartheta$ applied to the morphism $\vartheta_V$.) In cash, we have for every $g,g'\in G$ and every $v\in V$
$$
\vartheta^{(2)}_V(g\otimes g'\otimes v)=
[g]_H\otimes [gg']_H\otimes gg'v\,.
$$
Finally, we need to make the multiplication $\mu':\bfA\!'\circ\bfA\!'\to \bfA\!'$ more explicit. By Remark~\ref{rem:KE}, it is given by $\mu'_V=\Ind(\eps'_{\Res V}):\kk G\otimes_{\kk H}(\kk G\otimes_{\kk H}V)\to \kk G\otimes_{\kk H}V$. By~\eqref{eq:eta-eps'}, we have for every $g,g'\in G$ and $v\in V$
$$
\mu'_V(g\otimes g'\otimes v)=
\left\{\begin{array}{cc}g\otimes g'v=gg'\otimes v &\text{if }g'\in H
\\[.5em]
0&\text{if }g'\notin H.
\end{array}\right.
$$
With all morphisms and all actions being now explicit, it is direct to check commutativity of~\eqref{eq:ring-hom}. We leave this computation to the reader. (For verification, the two compositions send $g\otimes g'\otimes v$ to $[g]_H\otimes gg'v$ if $g'\in H$ and to 0 otherwise.)
\end{proof}

We can now replace the monad $\bfA\!'=\Ind_H^G\Res^G_H$ on $\cat C(G)=\kk G\MMod$ by the ring object $A^G_H=\kk(G/H)$. This actually changes slightly the result in that $\Res$ is not \emph{equal} to extension-of-scalars on the nose but only naturally isomorphic to it.

\begin{Thm}
\label{thm:fin-ind}%
Let $H\leq G$ be a finite-index subgroup and recall the ring object~$A=A^G_H$ of Definition~\ref{def:AH}. Then there is an equivalence of categories $\Psi:\cat C(H)\isotoo A\MModcat{C(G)}$ making the following diagram commute up to isomorphism\,:
$$
\xymatrix@C=1em@R=2em{
& \cat C(G) \ar[ld]_(.6){\Displ \Res} \ar[rd]^(.6){\Displ F_{A}}
\\
\cat C(H) \ar[rr]^-{\cong}_-{\Displ\Psi}
&& A\MModcat{C(G)}\,.
}
$$
Explicitly, the functor $\Psi$ is given as follows. For every
$\kk H$-module~$W$, we have
$$
\Psi(W)=\big(V'\,,\,\varrho'_W\big)
$$
for $V':=\Ind_H^G(W)=\kk G\otimes_{\kk H}W$, equipped with the following $A$-action $\varrho'_W$ in~$\cat C(G)$
$$
\varrho'_W:\ A\otimes V'=\kk(G/H)\otimes_\kk (\kk G\otimes_{\kk H} W)\too
\kk G\otimes_{\kk H} W=V'
$$
given for every $\gamma\in G/H$, $g\in G$ and $w\in W$ by
$$
\varrho'_W(\gamma\otimes g\otimes w)=
\left\{\begin{array}{cc}
g\otimes w&\text{if }g\in \gamma
\\[.5em]
0&\text{otherwise.}
\end{array}\right.
$$
The isomorphism $\Psi\circ\Res\isoto F_A$ of functors from $\cat C(G)$ to $A\MModcat{C(G)}$ is given for every $\kk G$-module~$V$ by the classical isomorphism $\vartheta_V$ of~\eqref{eq:theta} as follows\,:
$$
\Psi\circ\Res(V)=\Ind\Res(V)=\kk G\otimes_{\kk H}V\underset{\vartheta_V}\isotoo
\kk(G/H)\otimes_kV=A\otimes V=F_A(V)\,.
$$
\end{Thm}

\begin{proof}
Contemplate the following diagram\,:
$$
\vcenter{\xymatrix@C=4em@R=2em{
& \cat C(G)
 \ar[ld]_-{\Displ \Res}
 \ar[d]^-{\Displ F_{\bfA\!'}}
 \ar[rd]^(.6){\Displ F_{A}}
\\
\cat C(H)
 \ar[r]_-{\cong}^-{E'} \ar@/_2.5em/[rr]_-{\cong}^-{\Displ\Psi}
& \bfA\!'\MModcat{C(G)} \ar[r]_-{\cong}^-{\Theta}
& A\MModcat{C(G)}\,.
}}
$$
Here $E'$ is the Eilenberg-Moore functor associated to the $\Res/\Ind$ adjunction~\eqref{eq:res-ind}. We have seen in Corollary~\ref{cor:para-fin} that $E'$ is an equivalence and we know that $E'\circ \Res=F_{\bfA\!'}$. On the other hand, we have seen in Proposition~\ref{prop:AH}\,(c) that the monad $\bfA\!'$ is isomorphic, as a monad, to $A\otimes-$ via $\vartheta:\bfA\!'\isoto A\otimes-$. This induces an obvious isomorphism on the categories of modules $\Theta:=(\vartheta\inv)^*:\bfA\!'\MModcat{C} \isoto  A\MModcat{C}$
$$
\Theta\big(\,X\,,\,\varrho:\bfA\!'(X)\to X\big):=\big(\,X\,,\,\varrho\circ\vartheta_X\inv:A\otimes X\to X\big)\,.
$$
The natural isomorphism $\vartheta_V:\bfA\!'(V)\isoto A\otimes V$, for $V\in \cat C(G)$, defines a natural isomorphism $\vartheta\,:\ \Theta \circ F_{\bfA\!'}\isoto F_A$ of functors from $\cat C(G)$ to $A\MModcat{C(G)}$, using that $\vartheta$ is an isomorphism of monads (Proposition~\ref{prop:AH}). We now define $\Psi:=\Theta\circ E'$ to get the result. In particular, $\Psi\circ\Res=\Theta\circ E'\circ\Res=\Theta\circ F_{\bfA\!'}\underset{\vartheta}{\isoto} F_A$. The explicit formula for $\Psi=\Theta\circ E'$ given in the statement is immediate from the definition of $E'$ (see~\eqref{eq:E} in Remark~\ref{rem:KE}), the formula for $\eps'_W$ in~\eqref{eq:eta-eps'}, the above formula for $\Theta$ and finally the formula for $\vartheta\inv$ in~\eqref{eq:theta-inv}. The verification is now pedestrian.
\end{proof}

%------------------------------------------------------------------------------
%\goodbreak
\medbreak
\section{Variations on the theme and comments}
\label{se:variations}%
\medbreak
%------------------------------------------------------------------------------

The results of Section~\ref{se:ResGH} are not specific to the category $\kk G\MMod$ and hold for the derived $\Der(\kk G\MMod)$ and the stable $\kk G\SStab$ ones as well. But let us be careful, as the following example should warn us\,: We cannot naively ``derive" monadicity.

\begin{Rem}
\label{rem:warn}%
Consider an adjunction $F:\,\cat A\adj\cat B\,:G$ of abelian categories whose counit $\eps:FG\to \Idcat{B}$ is naturally split and such that the derived adjunction $\LF:\Der(\cat A)\adj \Der(\cat B)\,:\RG$ exists. Then it is not true in general that this derived adjunction has split counit, nor that it is monadic. For example, take $I\subset B$ an ideal of a commutative ring~$B$ such that $\iota_*:\Der(B/I)\to \Der(B)$ is not faithful (\eg\ $B=\bbZ$ and $I=4\bbZ$). On modules, the adjunction $B/I\otimes_B-:\, B\MMod\adj(B/I)\MMod\,:\iota_*$ has split counit (even an isomorphism). However $R\iota_*=\iota_*$ is not faithful, which means the derived adjunction cannot be monadic (the forgetful functor \emph{is} faithful) and in particular the derived counit cannot be split (Lemma~\ref{lem:monadicity}).

The problem is the following. Assume for simplicity, as in the above example, that $G$ is exact, so $\RG=G$ is just $G$ degreewise. For every $W_\sbull\in\Der(\cat B)$ choose an $F$-acyclic complex $P_\sbull$ and a quasi-isomorphism $P_\sbull\oto{s}G(W_\sbull)$ in~$\cat B$. Then the counit of the derived adjunction at~$W_\sbull$ is given by the composite $\eps\circ F(s)$\,:
$$
\xymatrix@C=4em{
\LF\circ \RG\,(W_\sbull)=\LF(G\,(W_\sbull))=\kern-4.4em
& F(P_\sbull)\ar[r]^-{\Displ F(s)}
& FG(W_\sbull)\ar[r]^-{\Displ \eps}  \ar@{-->}@/^1.5em/[l]_-{??}
& W_\sbull \ar@/^1.5em/[l]_-{\exists}
}
$$
where the last morphism is the original counit~$\eps$ applied degreewise. So, we see that if the original counit $\eps$ has a natural section, then we can use this section degreewise to split the last morphism above but we cannot split~$F(s)$ in general.

In our case, it is therefore essential that we do not need to left-derive $\Res^G_H$ but can simply use it degreewise on complexes. In that case, the above $F(s)$ is an isomorphism (the identity) and the derived counit is as split as the original one. Actually, all functors $\Res$, $\Ind$, $\CoInd$ (and $\otimes_\kk$ when $\kk$ is a field) are exact here.
\end{Rem}

The above counter-example explains the importance of the following easy result\,:

\begin{Lem}
\label{lem:der}%
Let $F:\,\cat A\adj\cat B\,:G$ be an adjunction of \emph{exact} functors between abelian categories (resp.\ Frobenius abelian categories), whose counit has a natural section. Then the induced adjunction $F:\Der(\cat A)\adj\Der(\cat B):G$ on derived (resp.\ $F:\underline{\cat A}\adj\underline{\cat B}:G$ on stable) categories also has a naturally split counit. Dually, if the unit has a natural retraction then so does the derived (resp.\ stable) unit.
\end{Lem}

\begin{proof}
The derived functors $F$ and $G$ are simply defined as $F$ and $G$ degreewise and so are the unit and counit of the derived adjunction. Therefore, we can define the section of the counit degreewise as well. The case of the stable categories is even simpler once we observe that $F$ (resp.\ $G$), as left (resp.\ right) adjoint of an exact functor, will preserve projective (resp.\ injective) objects.
\end{proof}

We therefore obtain the derived and stable analogues of Theorem~\ref{thm:paradigm}\,:

\begin{Thm}
\label{thm:para-der}%
Let $H\leq G$ be an arbitrary subgroup. Let us change notation and set $\cat C(G):=\Der(\kk G\MMod)$ the derived category of $\kk G$-modules. Then, the statement of Theorem~\ref{thm:paradigm} holds verbatim,
\ie restriction to~$H$ becomes extension-of-scalars with respect to the monad~$\bfA=\CoInd\Res$ on $\cat C(G)$. The same result holds if $\cat C(G)$ stands for $\kk G\SStab$ when $G$ is finite and $\kk$ is a field.
\end{Thm}

\begin{proof}
Same proof as Theorem~\ref{thm:paradigm}\,: The counit of the derived or stable adjunctions remains naturally split (Lemma~\ref{lem:der}) and we can then apply Lemma~\ref{lem:monadicity}.
\end{proof}

Similarly, when $\Index{G}{H}<\infty$ and $\kk$ is a field, the ring object $A=A^G_H$ of Definition~\ref{def:AH} will exist in every new tensor category which receives $\kk G\MMod$. Hence we obtain analogues of Theorem~\ref{thm:fin-ind}, like for instance\,:
\begin{Thm}
\label{thm:f-i-der}%
Let $H\leq G$ be a finite-index subgroup, $\kk$ a field and $\cat C(G)=\Der(\kk G\MMod)$. Consider $A=A^G_H$ in~$\cat C(G)$, as a complex concentrated in degree zero. There exists an equivalence of categories $\Psi:\cat C(H)\isotoo A\MModcat{C(G)}$ making the following diagram commute up to
isomorphism\,:
$$
\xymatrix@C=1em@R=2em@C=2em{
& \cat C(G) \ar[ld]_(.6){\Displ \Res} \ar[rd]^(.6){\Displ F_{A}}
\\
\cat C(H) \ar[rr]_-{\cong}^-{\Displ\Psi}
&& A\MModcat{C(G)}\,.
}
$$
The same result holds if $\cat C(G)$ stands for $\kk G\SStab$, when $G$ is finite.
\end{Thm}

\begin{proof}
The proof of Theorem~\ref{thm:fin-ind} holds verbatim with the new $\cat C(G)$. Again, all functors being exact, we can apply all (plain) isomorphisms in sight degreewise to obtain the necessary isomorphisms at the derived level.
\end{proof}

\begin{Rem}
The same results hold for finitely generated modules and bounded complexes as well, when they make sense, \ie when the restriction-(co)induction adjunction preserves those categories. In particular, for $G$ finite, we have
$$\kk H\sstab\cong A\MMod_{\kk G\sstab}.$$
\end{Rem}

\begin{Rem}
\label{rem:Conj}%
Let $G$ be a finite group and $X$ a finite left $G$-set. Consider the ring object $A_X:=\kk X$ in $\kk G\mmod$ with all $x\in X$ being orthogonal idempotents ($x\cdot x=x$ and $x\cdot x'=0$ for $x\neq x'$). This ring object is isomorphic to the sum of our $A^G_{H_i}$ (Definition~\ref{def:AH}) for the decomposition of $X\simeq G/H_1\sqcup\cdots\sqcup G/H_n$ in $G$-orbits. These commutative separable ring objects $A_X$ migrate to any $\otimes$-category receiving $\kk G\mmod$, like $\kk G\sstab$. Beyond these, every finite separable field extension~$\kk'/\kk$, with trivial $G$-action, would also define such commutative separable ring objects.

So, assume for simplicity that $\kk$ is a separably closed field. Up to isomorphism, the only commutative separable ring objects $A$ in the plain category $\kk G\mmod$ are the above~$A_X$. This remark was first made with Serge Bouc and Jacques Th\'evenaz. Indeed, since $\kk$ is separably closed, the underlying $\kk$-algebra of $A$ is isomorphic to $\kk\times\cdots \times\kk=\kk X$ for a set~$X$ of orthogonal primitive idempotents. But $G$ acts on $A$ by ring automorphisms, hence it permutes  idempotents, \ie $X$ inherits a $G$-action and therefore $A\simeq A_X$. It is tempting to ask whether the same holds in $\kk G\sstab$\,:
\end{Rem}

\begin{Que}
Let $\kk$ be separably closed and $A\in\kk G\sstab$ be a commutative separable ring object. Is there a finite $G$-set $X$ such that $A\simeq \kk X$ in $\kk G\sstab$?
\end{Que}

This problem is important for it asks whether the ``\'etale topology" which appears in modular representation theory is richer than what is produced by subgroups.

\begin{center}*\ *\ *\end{center}

We now give another example of a restriction functor which also satisfies monadicity. For this, let $\kk$ be a field of characteristic~$p>0$ and $G$ an elementary abelian $p$-group, \ie a product $G\simeq C_p\times\cdots \times C_p$ of copies of the cyclic group $C_p$ of order~$p$. A \emph{cyclic shifted subgroup} of $G$ is a ring homomorphism $\ell:\kk C_p\to \kk G$ such that $\kk G$ is flat (\ie free) as $\kk C_p$-module via~$\ell$. Cyclic shifted subgroups originate in Carlson's work on rank varieties; see~\cite{Carlson83}. Consider the stable category $\cat C(G)=\kk G\SStab$ and the adjunction $\Res_\ell:\cat C(G)\adj\cat C(C_p):\CoInd_\ell$ as in~\eqref{eq:adj}. As in Section~\ref{se:monadicity}, this adjunction induces a monad $\bfA_\ell=\CoInd_\ell\circ\Res_\ell$ on~$\cat C(G)$.

\begin{Thm}
\label{thm:scs}%
Let $\ell:\kk C_p\to \kk G$ be a cyclic shifted subgroup as above. Then the Eilenberg-Moore functor $E:\cat C(C_p)\too \bfA_\ell\MModcat{C(G)}$ is an equivalence such that $\ell^*=\Res_\ell$ coincides with the extension-of-scalars $F_{\bfA_\ell}:\cat C(G)\to \bfA_\ell\MModcat{C(G)}$.
\end{Thm}

\begin{proof}
By Theorem~\ref{thm:ancestor}\,(a), we simply need to prove that $\ell:\kk C_p\to \kk G$ has a retraction as $\kk C_p$-bimodule since the induced section of the counit passes from categories of modules to stable categories, where we can then apply Lemma~\ref{lem:monadicity}. Now, both rings are commutative, so it is enough to show that $\ell$ has a section as $\kk C_p$-module. By Nakayama, the $\kk$-vector space $\kk G/\textrm{Rad}(\kk C_p)$ admits a $\kk$-basis starting with the class of~1 (otherwise $\kk G=0$) and a lift of that basis in the free $\kk C_p$-module $\kk G$ gives a $\kk C_p$-basis starting with $1=\ell(1)$. Define $m:\kk G\to \kk C_p$ by $\kk C_p$-linear projection onto 1 and we have $m\circ \ell=\id$ as wanted.
\end{proof}

\begin{Rem}
By the above proof, the counit of the monad $\bfA_\ell$ on $\kk G\SStab$ is separable. So, $\kk C_p\SStab$ can be obtained out of $\kk G\SStab$ and $\bfA_\ell$ by the purely triangulated-categorical method of~\cite{Balmer11}. Now, the very simple category $\kk C_p\SStab$ might be conceptually understood as a \emph{field} in tensor triangular geometry although it is not yet clear what tensor-triangular fields exactly are, nor how to construct them in general. See~\cite[\S\,4.3]{BalmerICM}. Also, the monad $\bfA_\ell$ on $\kk G\SStab$ does not come from a ring object in $\kk G\SStab$ (unless $\ell$ is a plain subgroup). So, even in the study of finite groups, one should not discard monads in favor of ring objects.
\end{Rem}

\begin{Rem}
This result about cyclic shifted subgroups of elementary abelian groups probably extends to $\pi$-points of finite group schemes \`a la Friedlander-Pevtsova~\cite{FriedlanderPevtsova07}, at the cost of additional technicalities left to the interested readers.
\end{Rem}

\begin{center}*\ *\ *\end{center}

\begin{Rem}
\label{rem:desc}%
Mostly as a motivation for Part~\ref{part:B}, we discuss the attempt to use Theorem~\ref{thm:main} to unfold descent in its simplest form; see Knus-Ojanguren~\cite{KnusOjanguren74} or~\cite[\S\,3]{Balmer12}. Let $H\leq G$ be a finite-index subgroup and $A=A^G_H=\kk(G/H)$ the associated ring object in~$\kk G\MMod$ as in Definition~\ref{def:AH}. \textit{Suppose that the index $\Index{G}{H}$ is invertible in~$\kk$. Then $A$ satisfies descent in~$\kk G\MMod$. Similarly, if $\kk$ is a field, $A$ satisfies descent in $\Der(\kk G\MMod)$ and $\Db(\kk G\MMod)$. When moreover $G$ is finite, $A$ satisfies descent in $\kk G\mmod$, $\kk G\SStab$, $\kk G\sstab$, $\Db(\kk G\mmod)$.}

Indeed, define $\zeta:A=\kk(G/H)\to \kk$ by $\zeta(\gamma)=\Index GH\inv$ for every $\gamma\in G/H$. This $\zeta$ is a retraction of the unit $\eta:\unit\to A$. Consequently, we can apply~\cite[Cor.\,2.6]{Mesablishvili06a}, or the dual of Lemma~\ref{lem:monadicity} to get descent, \ie \emph{co}monadicity. In the triangulated cases, one can also use~\cite[\S\,3]{Balmer12}. This proves the above claim.

Conversely, for descent to hold, the index $\Index{G}{H}$ must be invertible in~$\kk$, at least for triangulated categories, like $\Db(\kk G\mmod)$ and beyond. Indeed, it is necessary that $A\otimes-$ be faithful. Now choose a distinguished triangle $\cdot \oto{\chi} \unit\overset{\eta}\to A\to \cdot$ featuring~$\eta$. Since $\eta$ is retracted after applying $A\otimes-$ (by~$\mu:A\otimes A\to A$), the morphism $\chi$ satisfies $A\otimes \chi=0$ hence is zero, by faithfulness of~$A\otimes-$. In a triangulated category, this forces $\eta$ to be retracted. A $\kk G$-linear retraction $\zeta:A=\kk(G/H)\to \kk=\unit$ must be given by $\gamma\mapsto u$ for all $\gamma\in G/H$, for some fixed $u\in \kk$. Since $\eta(1)=\sum\gamma$, we have $1=\zeta\circ\eta(1)=u\cdot|G/H|$. Hence $|G/H|\in\kk^\times$.

Descent for a commutative ring object $A$ in a tensor category~$\cat C$ asserts that $\cat C$ is equivalent to the descent category $\Desc_{\cat C}(A)$. The latter is described in terms of $A$-modules $X$ equipped with a \emph{gluing isomorphism} $\gamma:A\otimes X\isoto X\otimes A$ as $A\potimes 2$-modules satisfying the \emph{cocycle condition} ``$\gamma_2=\gamma_3\circ\gamma_1$" in the category of $A\potimes 3$-modules. Here $\gamma_i$ means ``$\gamma$ tensored with $\id_A$ in position~$i$". See~\cite[\S\,3]{Balmer12} for details. In the case of the category $\cat C=\cat C(G)$ depending on a group~$G$ (in any sense used above) and of the ring object $A=A^G_H$, we know by Theorem~\ref{thm:main} that the category of $A$-modules in~$\cat C(G)$ is equivalent to the category $\cat C(H)$. Using the Mackey formula, one can decompose $A\potimes 2$ as a sum of $A_{H^g\cap H}$ for $g$ in a chosen set of representatives of $\HGH$. Hence, by Theorem~\ref{thm:main} again, the category of $A\potimes 2$-modules in~$\cat C(G)$ is equivalent to the coproduct of the corresponding categories $\cat C(H^g\cap H)$. As usual, this suffers from the choice of representatives for $\HGH$. A similar, even less natural description can be made for $A\potimes 3$-modules in~$\cat C(G)$ in terms of categories of the form $\cat C(H^{g_2}\cap H^{g_1}\cap H)$ by a third application of our Theorem~\ref{thm:main} together with a double layer of Mackey formulas and more choices. However, the Mackey formulas become really messy when dealing with three factors and most annoyingly the (three) extensions from $A\potimes{2}$-modules to $A\potimes{3}$-modules cannot be all controlled by \emph{one} such set of choices. The reader without experience with those issues is invited to try for himself\,!

These technicalities require a more efficient formalism, as in Part~\ref{part:B} below.
\end{Rem}

%------------------------------------------------------------------------------
%------------------------------------------------------------------------------
\goodbreak
\part{Stacks of representations}
\label{part:B}%
\bigbreak
%------------------------------------------------------------------------------
Form now on, $G$ is assumed to be a finite group.
%------------------------------------------------------------------------------
\medbreak
\section{A Grothendieck topology on finite $G$-sets}
\label{se:top}%
\medbreak
%------------------------------------------------------------------------------

For Grothendieck topologies, we follow Mac\,Lane-Moerdijk~\cite[Chap.\,III]{MacLaneMoerdijk94}. See also SGA\,4\cite{SGA4}, Kashiwara-Schapira~\cite[Chap.\,16]{KashiwaraSchapira06} or Vistoli~\cite{Vistoli05}.

\begin{Not}
Let $\Gsets$ be the category of finite left $G$-sets, with $G$-equivariant maps (``$G$-maps" for short).
\end{Not}

\begin{Rem}
\label{rem:lim}%
The category $\Gsets$ has finite limits, in particular pull-backs
$$
\kern10em\xymatrix@R=1.5em{
\kern-16em\SET{(x,y)\in  X\times Y\,}{\,\alpha(x)=\beta(y)}=X\times_Z Y
 \ar[r]^-{\pr_2} \ar[d]_-{\pr_1}
& Y \ar[d]^{\beta}
\\
X \ar[r]^-{\alpha} & Z}
$$
where $G$ acts diagonally on $X\times Y$ and $X\times_ZY$. If $X\simeq \sqcup_{i=1}^nX_i$ in $G$-sets (for instance, if $X_1,\ldots,X_n$ are the $G$-orbits of~$X$) then $X\times_ZY\simeq\sqcup_{i=1}^n(X_i\times_ZY)$.
\end{Rem}

\begin{Not}
As usual, we denote by $\con{g}h=g hg\inv$ and $h^g=g\inv h\,g$ the conjugates of $h\in G$ by $g\in G$ and similarly for $\con{g}H$ and $H^g$ for a subgroup $H\leq G$.
\end{Not}

We shall need a couple of Mackey formulas, in the following generality\,:
\begin{Prop}[Mackey formula]
\label{prop:Mackey}%
Let $K_1,K_2\leq H\leq G$ be subgroups. Let $S\subset H$ be a set of representatives of $\doublequot{K_1}{H}{K_2}$, meaning that the composite $S\hook H\onto \doublequot{K_1}{H}{K_2}$ is bijective. Then we have a bijection of $G$-sets
\begin{equation}
\label{eq:Mackey}%
\begin{aligned}
\coprod_{\Displ t\in S} G/(K_1^t\cap K_2) & \isotoo (G/K_1)\times_{G/H} (G/K_2)
\\
[z]_{K_1^t\cap K_2} & \longmapsto \big(\,[zt\inv]_{K_1}\,,\,[z]_{K_2}\big)
\end{aligned}
\end{equation}
where the notation $[-]$ indicates classes in the relevant cosets (as in Part~\ref{part:A}).
\end{Prop}

\begin{proof}
This is well-known. Use the surjection $(G/K_1)\times_{G/H} (G/K_2)\onto \doublequot{K_1}{H}{K_2}$ given by $([z_1]_{K_1},[z_2]_{K_2})\mapsto {}_{K_1}[z_1\inv\,z_2]_{K_2}$ and show that the fiber of each ${}_{K_1}[t]_{K_2}$ is exactly given by $G/(K_1^t\cap K_2)$ via the above map.
\end{proof}

\begin{Cor}
\label{cor:Mackey}%
Let $H$ be a finite group. Let $H',K\leq H$ be two subgroups such that the index $\Index{H}{K}$ is prime to~$p$. For any set of representatives $S\subset H$ of $\doublequot KH{H'}$, there exists $t\in S$ such that the index $\Index{H'}{K^t\cap H'}$ is also prime to~$p$.
\end{Cor}

\begin{proof}
Let $v\in\bbN$ be such that $p^v$ divides $\Index{H}{H'}$ but $p^{v+1}$ does not. Applying~\eqref{eq:Mackey} for $G=H$, $K_1=K$ and $K_2=H'$ and counting elements on both sides gives
$$
\sum_{t\in S}\,\Index{H}{K^t\cap H'}=\Index HK\cdot \Index H{H'}\,.
$$
Since $\Index{H}{K}$ is prime to~$p$, the right-hand side is not divisible by~$p^{v+1}$ . Hence $p^{v+1}$ cannot divide all the left-hand terms. Hence there exists $t\in S$ such that $p^{v+1}$ does not divide $\Index{H}{K^t\cap H'}$. Now, for that~$t$, contemplate the tower of subgroups $K^t\cap H'\ \leq \ H'\ \leq \ H$. We have that $p^{v+1}$ does not divide $\Index{H}{K^t\cap H'}$ but $p^v$ divides $\Index{H}{H'}$. Hence $\Index{H'}{K^t\cap H'}$ is prime to $p$, as wanted.
\end{proof}

\begin{Not}
The \emph{stabilizer} of $x\in X\in \Gsets$ is $\St_G(x):=\SET{g\in G}{gx=x}$. For every $G$-map $f:Y\to X$ and every~$y\in Y$, we have $\St_G(y)\leq \St_G(f(y))$ in~$G$.
\end{Not}

\begin{Def}
\label{def:top}%
Let $X$ be a finite $G$-set.
\begen[(a)]
\item
An arbitrary (possibly infinite) family of $G$-maps $\{U_i\otoo{\alpha_i} X\}_{i\in I}$ in $\Gsets$ is a \emph{sipp-covering} if for every $x\in X$ there exists $i\in I$ and $u\in U_i$ such that $\alpha_i(u)=x$ and such that the index $\Index{\St_G(x)}{\St_G(u)}$ is prime to~$p$.
\smallbreak
\item
A single morphism $U\oto{\alpha} X$ in $\Gsets$ is a \emph{sipp-cover} if $\{U\to X\}$ is a covering, \ie for every $x\in X$ there exists $u\in \alpha\inv(x)$ with $[\St_G(x):\St_G(u)]$ prime to~$p$\,:
$$
\xymatrix@C=1em@R=1em{
U \ar@{->>}[d]_-{\alpha} \ar@{}[r]|{\ni}
& u \ar@{|->}[d]
&& \St_G(u) \ar@{}[d]|{\mid\bigwedge}
& \ar@{}[d]|-{\quad\textrm{index prime to }p}
\\
X \ar@{}[r]|{\ni}
& x
&& \St_G(x)
&}
$$
\ened
After Theorem~\ref{thm:top}, we shall call this (the basis of) the \emph{sipp topology} on~$\Gsets$\,(\footnote{In the tradition of the fppf- and fpqc-topologies, the acronym ``sipp" really stands for the French\,: ``\underbar{s}tabilisateurs d'\underbar{i}ndice \underbar{p}remier \`a \underbar{$p$}".}).
\end{Def}

\begin{Rem}
A family $\{U_i\to X\}_{i\in I}$ is a sipp-covering if and only if there exists $U\in \Gsets$ and $U\to \coprod_{i\in I}U_i$ such that the composite $U\to X$ is a sipp-cover.
\end{Rem}

\begin{Exa}
\label{exa:top}%
Let $K\leq H$ be subgroups of~$G$. Then the projection $G/K\onto G/H$ is a sipp-cover if and only if the index $\Index{H}{K}$ is prime to~$p$. For a family $K_i\leq H$, $\{G/K_i\onto G/H\}_{i\in I}$ is a sipp-covering if and only if some $[H:K_i]$ is prime to~$p$.
\end{Exa}

\begin{Thm}
\label{thm:top}%
The sipp-coverings of Definition~\ref{def:top} form a basis of a \emph{Grothendieck topology} on~$\Gsets$, namely they satisfy all the following properties\,:
\begen
\item
Every isomorphism $U\isoto X$ is a sipp-cover.
\smallbreak
\item
If $\{\alpha_i:U_i\to X\}_{i\in I}$ is a sipp-covering of~$X$ and if $\beta:Y\to X$ is a $G$-map, then the pull-backs define a sipp-covering $\{\pr_2:U_i\times_XY\to Y\}_{i\in I}$ of~$Y$.
\smallbreak
\item
If $\{\alpha_i:U_i\to X\}_{i\in I}$ is a sipp-covering of~$X$ and if $\{\beta_{ij}:V_{ij}\to U_i\}_{j\in J_i}$ is a sipp-covering of $U_i$ for every $i\in I$, then the composite family $\{\alpha_i\beta_{ij}:V_{ij}\to X\}_{i\in I, j\in I_i}$ is a sipp-covering of~$X$.
\ened
In words, $\Gsets$ becomes a \emph{site} (a category equipped with a Grothendieck topology).
\end{Thm}

\begin{proof}
Parts~(a) and~(c) are easy exercises. Let us prove~(b). Using~(a) and distributivity of pull-back with coproducts (Remark~\ref{rem:lim}), it suffices to prove the following special case\,: Let $H$ be a subgroup of~$G$ and $K,H'\leq H$ two subgroups of~$H$ such that $\Index{H}{K}$ is prime to~$p$. Consider the right-hand square of $G$-sets below\,:
$$
\xymatrix@C=3em@R=1.8em{
\coprod_{t\in S}\ G/(K^t\cap H')\ \ar[r]^-{\simeq}_-{\eqref{eq:Mackey}}
&\ {G/K}\times_{G/H} {G/H'}\ar[r]^-{\Displ\pr_2} \ar[d] & {G/H'} \ar[d]^{\Displ\beta}
\\
&{G/K} \ar[r]^-{\Displ\alpha} & G/H\,.
}
$$
Here $\alpha$ is the sipp-cover and $\beta$ is the other map. We need to prove that $\pr_2$ is a sipp-cover. The left-hand isomorphism is Mackey's formula~\eqref{eq:Mackey} for any $S\subset H$ such that $S\isoto \doublequot{K}{H}{H'}$. Composing this isomorphism with $\pr_2$ gives us the $G$-map
$$
\xymatrix@C=4em{
{\coprod_{t\in S}^{\vphantom{I^I}}\ G/(K^t\cap H') }\ \ar[r]^-{\coprod \Displ\alpha_t} & \ G/H'}
$$
where $\alpha_t:G/(K^t\cap H')\onto G/H'$ is the projection associated to $K^t\cap H'\leq H'$. This is now a sipp-cover if at least one of the indices $\Index{H'}{K^t\cap H'}$ is prime to~$p$ (Example~\ref{exa:top}) and this is exactly what was established in Corollary~\ref{cor:Mackey}.
\end{proof}

\begin{Rem}
\label{rem:loc}%
An object $X\in \Gsets$ is sipp-local (meaning that for every sipp-covering of~$X$, one of its morphisms admits a section) if and only if $X$ is an orbit whose stabilizer is a $p$-subgroup of~$G$, \ie $X\simeq G/H$ with $H$ a $p$-group. Indeed, suppose that $|H|$ is a power of~$p$ and that $\{U_i\oto{\alpha_i} G/H\}_{i\in I}$ is a sipp-covering. By Example~\ref{exa:top}, some orbit of some $U_i$ must be isomorphic to $G/K$ with $K\leq H$ of index prime to~$p$, which forces $K=H$. Conversely, suppose that $X$ is sipp-local. Then $X$ is connected (if $X=X_1\sqcup X_2$, use the covering $\{X_i\hook X\}_{i=1,2}$), hence $X\simeq G/H$ for some $H\leq G$. Let $K\leq H$ be a Sylow $p$-subgroup of $H$ and consider the cover $G/K\onto G/H$. This has a section, hence $K=H$ and $H$ is a $p$-group.
\end{Rem}

\begin{Rem}
\label{rem:sheaf}%
With this Grothendieck topology, we can now speak of \emph{sipp-sheaves} on $\Gsets$. A presheaf of sets, \ie a functor $P:\Gsets\op\to \Sets$, is a \emph{sheaf} if for every covering $\{U_i\to X\}_{i\in I}$ the following usual sequence of sets is an equalizer\,:
\begin{equation}
\label{eq:sheaf}%
\vcenter{\xymatrix@C=2.2em{
P(X) \ \ar@{ >->}[rr]^-{\prod_{i} P(\alpha_i)}
&& \ \Displ\prod_{i\in I}P(U_i)\
 \ar@<.5em>[rrr]^-{\prod_{j,k}P(\pr_1)\circ\,\pr_j}
 \ar[rrr]_-{\prod_{j,k}P(\pr_2)\circ\,\pr_k}
&&& \Displ\prod_{(j,k)\in I\times I} P(U_j\times_X U_k)
}}\kern-1em
\end{equation}
This means that restriction $s\mapsto (\alpha_i^*(s))_{i\in I}$ yields a bijection between $P(X)$ and the subset of those $(s_i)_{i\in I}\in\prod_{i}P(U_i)$ such that $\pr_1^*(s_j)=\pr_2^*(s_k)$ for every $j,k\in I$, where $\pr_1:U_j\times_XU_k\to U_j$ and $\pr_2:U_j\times_XU_k\to U_k$ are the two projections. Here $\alpha_i^*=P(\alpha_i)$ and $\pr_i^*=P(\pr_i)$ are the ``restriction" maps for the presheaf~$P$.
\end{Rem}

\begin{Rem}
\label{rem:cover}%
The sipp topology is \emph{quasi-compact}, in the sense that for every sipp-covering $\{U_i\to X\}_{i\in I}$, there exists $J\subseteq I$ finite such that $\{U_i\to X\}_{i\in J}$ is a covering too. Hence, it suffices to check sheaf conditions~\eqref{eq:sheaf} for \emph{finite} coverings $\{U_i\oto{\alpha_i} X\}_{i=1}^n$. Furthermore, $\Gsets$ has finite coproducts $\coprod_{i=1}^n U_i$ and finite coverings as above induce covers $U:=\coprod_{i=1}^n U_i\otoo{\sqcup_i\alpha_i}X$. Hence it suffices to verify sheaf conditions for sipp-covers $\alpha:U\to X$. This reduction from coverings $\{U_i\to X\}_{i\in I}$ to a single morphism $U\to X$ is a well-known flexibility of Grothendieck topologies.
\end{Rem}

\begin{Rem}
We do not use but simply indicate that our sipp topology is \emph{subcanonical}, \ie every represented presheaf $\Hom_\Gsets(-,Z): \Gsets\op\to \Sets$ is a sheaf, for every $Z\in\Gsets$. This follows from surjectivity of sipp-covers $U\onto X$.
\end{Rem}

\begin{Prop}
\label{prop:sheaf}%
Let $P:\Gsets\op\to \Sets$ be a presheaf. Then it is a sheaf if and only if the following conditions are satisfied\,:
\begen[(i)]
\item
Whenever $X=X_1\sqcup X_2$, the natural map $P(X)\to P(X_1)\times P(X_2)$ is an isomorphism. Also $P(\varnothing)$ is one point.
\item
For every pair of subgroups $K\leq H$ in $G$, such that $[H:K]$ is prime to~$p$, the sheaf condition~\eqref{eq:sheaf} holds for the sipp-cover $G/K\onto G/H$.
\ened
\end{Prop}

\begin{proof}
These conditions are easily seen to be necessary. Conversely, suppose that~(i) holds, then we can reduce the verification of the sheaf condition for all covers $U\to X$ (Remark~\ref{rem:cover}) to covers of the orbits of~$X$, so we can assume that $X=G/H$ for some subgroup $H\leq G$. In that case, the cover admits a refinement of the form $G/K\to G/H$ where we have $\Index{H}{K}$ prime to~$p$ and we can apply condition~(ii).
\end{proof}

Here is an amusing and yet useful example. For every $X\in\Gsets$, let $\overline{X}:=\SET{Gx}{x\in X,\ p\textrm{ divides }|\St_G(x)|}$ be the set of $G$-orbits of points of~$X$ with stabilizer of order divisible by~$p$. Then $\overline{(-)}$ is a well-defined functor from $\Gsets$ to finite sets, since $G$-maps only enlarge stabilizers and preserve $G$-orbits.
\begin{Prop}
\label{prop:uA}%
Let $A$ be an abelian group. Define the abelian group $\uA(X)$ to be $A^{\overline{X}}=\Mor_{\Sets}(\overline{X},A)$ for every $X\in\Gsets$ and the homomorphism $\uA(\alpha):\uA(X)\to \uA(Y)$ to be $A^{\overline{\alpha}}=\overline\alpha^*=-\circ\overline\alpha$ for every $G$-map $\alpha:Y\to X$. Then the presheaf $\uA:\Gsets\op\to \bbZ\MMod$ is a sheaf of abelian groups for the sipp topology.
\end{Prop}

\begin{proof}
To check that $\uA$ is a sheaf, by Proposition~\ref{prop:sheaf}, it suffices to verify the sheaf condition for covers of the form $U=G/K\onto G/H=X$ with $K\leq H$ of index prime to~$p$. This is clear for then $\overline{U}=\overline{X}$ and the two maps $\overline{U\times_XU}\rightrightarrows \overline{U}$ are equal.
\end{proof}

\begin{Rem}
\label{rem:uA}%
Indeed, $\uA$ is the sipp-sheafification of the constant presheaf~$A$. We call it the \emph{constant sheaf} associated to~$A$. For every $X\simeq G/H_1\sqcup \cdots \sqcup G/H_n$ we have $\uA(X)=A^m$ where $m=\#\SET{1\leq i\leq n}{p\textrm{ divides }|H_i|}$. The behavior of $\uA(-)$ on $G$-maps is rather obvious and only involves 0 or $\id_A$ on each component~$A$.
\end{Rem}

%------------------------------------------------------------------------------
\medbreak\goodbreak
\section{Plain, derived and stable representations of $G$-sets}
\label{se:rep}%
\medbreak
%------------------------------------------------------------------------------

Recall that $G$ is a finite group. We want to define the category of representations $\Rep(X)$ for every finite $G$-set~$X$ in such a way that $\Rep(G/H)$ is equivalent to $\kk H\mmod$. The problem with using $\kk H\MMod$ directly is that ``twisted" restriction $\con{g}(-)\dto^H_K$ from $H$ to $K$ \emph{does} depend on the choice of the representative~$g$ in its left $H$-class ${}_{H}[g]\in {}_{H}\backslash N_G(K,H)\cong\Hom_\Gsets(G/K\,,\,G/H)$. The standard trick around this indeterminacy is to use the associated ``action groupoids" as follows.

\begin{Def}
\label{def:groupoid}%
Let $X$ be a (finite) $G$-set. Define the \emph{action groupoid} $G\ltimes X$ to be the category whose objects are the elements of~$X$ with morphisms $\Mor_{G\ltimes X}(x,x'):=\SET{g\in G}{gx=x'}$, being subsets of~$G$. Composition is defined by multiplication in~$G$. Clearly, every morphism in $G\ltimes X$ is an isomorphism, \ie $G\ltimes X$ is a groupoid. For every $G$-map $\alpha:X\to Y$, the functor $G\ltimes\alpha\,:\ G\ltimes X\to G\ltimes Y$ is simply $\alpha$ on objects and the ``inclusion" on morphisms (as subsets of~$G$).
\end{Def}

We can now speak of representations, as usual.

\begin{Def}
\label{def:rep}%
Let $\cat A$ be a fixed ``base" additive category (\eg\ $\cat A=\kk\MMod$ for a commutative ring~$\kk$). For every $G$-set $X$, denote by
$$
\Rep(X)=\cat A^{G\ltimes X}
$$
the category of functors from $G\ltimes X$ to~$\cat A$. We call it \emph{the (plain) category of representations of~$X$ (in~$\cat A$)}. See Remark~\ref{rem:cash} for a more elementary approach.

Assume moreover that $\cat A$ is abelian. Then so is~$\Rep(X)$. Let then $\Der(\Rep(X))$ be the \emph{derived category of representations}, whose objects are complexes in $\Rep(X)$ and morphisms are morphisms of complexes with quasi-isomorphisms inverted.

If we assume that $\kk$ is a field and $\cat A=\kk\MMod=:\kk\VVect$, then we claim that $\Rep(X)$ is a Frobenius category, meaning that injective and projective objects coincide and there are enough of both. We can therefore construct the \emph{stable category of representations} $\StRep(X)=\Rep(X)/\Proj(\Rep(X))$ as the additive quotient by the projective objects. It has the same objects as $\Rep(X)$ but any two morphisms whose difference factors via a projective are identified.

Both $\Der(\Rep(X))$ and $\StRep(X)$ are well-known triangulated categories.
\end{Def}

\begin{Rem}
\label{rem:cash}%
Removing groupoids from the picture, an object $V$ of $\Rep(X)$ consists of the data of objects $V_x$ in~$\cat A$, for every $x\in X$, together with isomorphisms $V_g:V_x\isoto V_{gx}$ in~$\cat A$, for every $g\in G$, subject to the rule that $V_1=\id$ and $V_{g_2g_1}=V_{g_2}\circ V_{g_1}$. A morphism $f:V\to V'$ in $\Rep(X)$ consists of a collection $f_x:V_x\to V'_x$ of morphisms in $\cat A$, for every $x\in X$, such that $V'_g\circ f_x=f_{gx}\circ V_g$ from $V_x$ to $V'_{gx}$, for every $g\in G$. Composition is the obvious $(f'\circ f)_x=f'_x\circ f_x$.
\end{Rem}

\begin{Exa}
\label{exa:BH}%
Let $H\leq G$ be a subgroup and $G/H\in\Gsets$ the associated orbit. Then the groupoid $\rmB H=H\ltimes\ast$ (with one object $\ast$ and $H$ as automorphism group) is equivalent to $G\ltimes (G/H)$ via $\iota_H:\rmB H\hook G\ltimes(G/H)$, $\ast\mapsto [1]_H$. Let $\cat A=\kk\MMod$ for a commutative ring~$\kk$. Then we have an equivalence of categories
\begin{equation}
\label{eq:iota}%
\iota_H^*:\Rep(G/H)\isotoo \kk H\MMod
\end{equation}
given by $V\mapsto V_{[1]}$. In particular, when $\kk$ is a field, $\Rep(G/H)$ is a Frobenius abelian category and, since $\Rep(X_1\sqcup X_2)\cong\Rep(X_1)\oplus\Rep(X_2)$, the category $\Rep(X)$ is Frobenius for every finite $G$-set~$X$ (as claimed in Definition~\ref{def:rep}).
\end{Exa}

Let us clarify the functoriality of $\Rep(X)$ in the $G$-set~$X$.

\begin{Def}
Let $\alpha:Y\to X$ be a morphism of $G$-sets. Let $\alpha^*:\Rep(X)\to \Rep(Y)$ be the functor $\cat A^{G\ltimes \alpha}=-\circ(G\ltimes \alpha)$. For every representation $V\in \Rep(X)$, we can give $\alpha^*V\in\Rep(Y)$ as in Remark~\ref{rem:cash}\,: For every $y\in Y$ and $g\in G$, we have
$$
(\alpha^*V)_y=V_{\alpha(y)}
\qquadtext{and}
(\alpha^*V)_g=V_g\,.
$$
Similarly, for every $f:V\to V'$ over~$X$, we have $(\alpha^* f)_y=f_{\alpha(y)}$ for every $y\in Y$.
\end{Def}

\begin{Rem}
The above functor $\alpha^*:\Rep(X)\to\Rep(Y)$ is exact when $\cat A$ is abelian and preserves projective objects when $\cat A=\kk\VVect$ (Example~\ref{exa:BH}). This induces well-defined functors on derived and stable categories, still denoted
$$
\alpha^*:\Der(\Rep(X))\to \Der(\Rep(Y))\qquadtext{and}\alpha^*:\StRep(X))\to \StRep(Y))\,.
$$
\end{Rem}

\begin{Prop}
We have a strict contravariant functor $\Rep(-):\Gsets\op\too\Add$ from $G$-sets to the category of additive categories. Similarly, when $\cat A$ is abelian (resp.\ when $\cat A=\kk\VVect$ for a field~$\kk$) then $\Der(\Rep(-)):\Gsets\op\too\Add$ (resp.\ $\StRep(-):\Gsets\op\too\Add$) is also a strict contravariant functor.
\end{Prop}

\begin{proof}
Easy verification. The new functors $\alpha^*$ are given by the same formula as above, applied objectwise. Note in particular that we do not need to derive~$\alpha^*$.
\end{proof}

\begin{Rem}
\label{rem:no-iso}%
For $Z\oto{\beta}Y\oto{\alpha}X$, we have $(\beta\circ\alpha)^*=\alpha^*\circ\beta^*$ on the nose, not up to isomorphism. Hence $\Rep(-)$, $\Der(\Rep(-))$ and $\StRep(-)$ are \emph{strict} functors, not mere pseudo-functors. Of course, having only pseudo-functors would not be a big problem, since stack theory is tailored for pseudo-functors (or fibered categories), but this nice fact reduces the amount of technicalities below.
\end{Rem}

We now unfold the right adjoints $\alpha_*$ to the above functors $\alpha^*$.

\begin{Def}
\label{def:adj}%
Let $\alpha:Y\to X$ be a $G$-map. Let $W\in \Rep(Y)$ be a representation of~$Y$. As in Remark~\ref{rem:cash}, define a representation $\alpha_*W$ over~$X$ by
\begin{equation}
\label{eq:alpha_*}%
(\alpha_*W)_x=\prod_{y\in\alpha\inv(x)}W_y
\qquadtext{and}
(\alpha_*W)_g=\prod_{y\in\alpha\inv(x)}W_g\quad\textrm{(diagonally)}
\end{equation}
for every $x\in X$ and every $g\in G$. For a morphism $f:W\to W'$ over~$Y$, we define $\alpha_*f:\alpha_*W\to \alpha_*W'$ by $(\alpha_*f)_x=\prod\limits_{y\in\alpha\inv(x)}f_y$ (diagonally) for every~$x\in X$.
\end{Def}

\begin{Rem}
\label{rem:adj}%
The above product is simply a direct sum, since $\cat A$ is assumed additive. However, the product is the right concept here if we drop the assumption that our $G$-sets are finite. In that case, one should assume that $\cat A$ has small products.
\end{Rem}

\begin{Prop}
\label{prop:adj}%
Let $\alpha:Y\to X$ be a $G$-map. Then we have three adjunctions
$$
\xymatrix@R=1.5em{
\Rep(X) \ar@<-.3em>[d]_(.47){\Displ\alpha^*}
&& \Der(\Rep(X)) \ar@<-.3em>[d]_(.47){\Displ\alpha^*}
&& \StRep(X) \ar@<-.3em>[d]_(.47){\Displ\alpha^*}
\\
\Rep(Y) \ar@<-.3em>[u]_(.47){\Displ\alpha_*}
&& \Der(\Rep(Y)) \ar@<-.3em>[u]_(.47){\Displ\alpha_*}
&& \StRep(Y) \ar@<-.3em>[u]_(.47){\Displ\alpha_*}
}
$$
For the plain one, the unit $\eta\pp{\alpha}:\Id_{\Rep(X)}\to\alpha_*\alpha^*$ is given by the formula
$$
\begin{aligned}
(\eta\pp{\alpha}_V)_x\,:\ V_x & \too \prod_{y\in\alpha\inv(x)}\kern-.5em V_{x}\quad=(\alpha_*\alpha^* V)_x
\\
v & \longmapsto (v)_{y\in\alpha\inv(x)}\qquad \textrm{(constant)}
\end{aligned}
$$
for every $V\in\Rep(X)$ and $x\in X$, whereas the counit $\eps\pp{\alpha}:\alpha^*\alpha_*\to\Id_{\Rep(Y)}$ is given for every $W\in \Rep(Y)$ and  $y\in Y$ by
$$
\begin{aligned}
(\eps\pp{\alpha}_W)_y\,:\ (\alpha^*\alpha_*W)_y=\prod_{y'\in\alpha\inv(\alpha(y))}\kern-1em W_{y'} \ & \too W_{y}
\\
(w_{y'})_{y'\in\alpha\inv(\alpha(y))} & \longmapsto w_y\,.
\end{aligned}
$$
The derived one (supposing $\cat A$ abelian) and the stable one (supposing $\cat A=\kk\VVect$ for a field~$\kk$) are induced by the plain one objectwise.
\end{Prop}

\begin{proof}
Verify the unit-counit relations, namely here $\eps\pp{\alpha}_{\alpha^*V}\circ \alpha^*(\eta\pp{\alpha}_V)=\id_{\alpha^* V}$ and $\alpha_*(\eps\pp{\alpha}_W)\circ \eta\pp{\alpha}_{\alpha_*W}=\id_{\alpha_*W}$. See~\cite[Chap.\,IV]{MacLane98} if necessary. For the derived and stable versions, the functors $\alpha^*$ and $\alpha_*$ are exact and we apply Lemma~\ref{lem:der}.
\end{proof}

%------------------------------------------------------------------------------
\medbreak
\section{Beck-Chevalley property and descent}
\label{se:descent}%
\medbreak
%------------------------------------------------------------------------------

Recall that $G$ is a finite group. In Section~\ref{se:rep}, we recalled the functor $\Rep(-):\Gsets\op\too\Add$ of plain representations, together with the derived $\Der(\Rep(-))$ and stable $\StRep(-)$ versions. Each of them is a \emph{presheaf} of categories on our site~$\Gsets$. These constructions did not involve the sipp topology of Section~\ref{se:top}. Saying that these presheaves of categories are \emph{stacks} heuristically means that they are \emph{sheaves} for the Grothendieck topology. To make this precise, we recall the basics of Grothendieck's descent formalism. A detailed reference is Vistoli~\cite{Vistoli05}.

\begin{Rem}
\label{rem:strict}%
The following definition is usually given for \emph{pseudo}-functors but we don't need this generality here as we have seen in Remark~\ref{rem:no-iso}. This happy simplification explains the word ``strict" below. Also, by Remark~\ref{rem:cover}, we restrict attention to covers $\cat U=\{U\to X\}$, \ie coverings with a single map.
\end{Rem}

\begin{Def}
\label{def:descent}%
Let $\cat U=\{U\oto{\alpha} X\}$ be a cover of $X$ in a site $\cat G$ with pull-backs and let $\cat D:\cat G\op\to \Cat$ be a (strict) contravariant functor from $\cat G$ to the category $\Cat$ of small categories. We denote by $U\pp{n}:=U\times_X\cdots\times_XU$ ($n$ factors). The \emph{(strict) descent category for the cover~$\cat U$}, denoted $\Desc_{\cat D}(\cat U)$, is defined as follows. Its objects are the \emph{(strict) descent data}, \ie pairs $(W,s)$ where $W$ is an object of $\cat D(U)$ and $s$ is a so-called \emph{gluing isomorphism}
$$
s\,:\ \pr_2^*W \isotoo \pr_1^* W
$$
in $\cat D(U\pp{2})$, where $\pr_i:U\pp{2}=U\times_XU\to U$, $i=1,2$, are the two projections, subject to the so-called \emph{cocycle condition}\,:
$$
\pr_{13}^*(s)=\pr_{12}^*(s)\circ\pr_{23}^*(s)
$$
in $\cat D(U\pp{3})$, where $\pr_{ij}\,:U\pp{3}=U\times_X U\times_X U\too U\pp{2}$ denotes projection on the $i^\textrm{th}$ and $j^\textrm{th}$ factors. A \emph{morphism of descent data} $f:(W,s)\to (W',s')$ is a morphism $f:W\to W'$ in $\cat D(U)$ such that $\pr_1^*(f)\circ s=s'\circ\pr_2^*(f)$ in $\cat D(U\pp{2})$.

There is a comparison functor $Q:\cat D(X)\too \Desc_{\cat D}(\cat U)$ mapping an object $V$ of $\cat D(X)$ to $\alpha^*V\in\cat D(U)$ together with the identity $s=\id:\pr_2^*\alpha^*V=\pr_1^*\alpha^*V$ as gluing isomorphism in $\cat D(U\pp{2})$. On morphisms, we set of course $Q(f)=\alpha^*(f)$.
\end{Def}

\begin{Def}
\label{def:stack}%
We say that the presheaf $\cat D:\cat G\op\to \Cat$ \emph{satisfies strict descent} with respect to the cover $\cat U$ of $X$ if the comparison functor $Q:\cat D(X)\to \Desc_{\cat D}(\cat U)$ is an equivalence of categories. We say that $\cat G$ is a \emph{strict stack} over the site~$\cat G$ if it satisfies strict descent with respect to every cover $\cat U$ of every object $X$ in~$\cat G$.
\end{Def}

\begin{Rem}
The descent property of $\cat D:\cat G\op\to\Cat$ with respect to a cover $\cat U$ of $X$ means two things\,: First $Q:\cat D(X)\to \Desc_{\cat D}(\cat U)$ is fully faithful and second it is essentially surjective. Full-faithfulness roughly says that morphisms in $\cat D(X)$ are sipp-sheaves and essential surjectivity says that every descent datum $(W,s)$ has a \emph{solution}, \ie an object $V\in \cat D(X)$ with an isomorphism $f:\alpha^*(V)\isoto W$ in $\cat D(U)$, compatible with the gluing isomorphisms on the ``intersection" $U\pp{2}$. Such a solution $V$ is then unique up to unique isomorphism in~$\cat D(X)$.
\end{Rem}

The following property is the key to reducing descent problems to comonadicity.

\begin{Def}
\label{def:BC}%
Let $\cat G$ be a category with pull-backs and let $\cat D:\cat G\op\to \Cat$ be a contravariant functor. Denote by $\alpha^*:\cat D(X)\to \cat D(Y)$ the functor associated to $\alpha:Y\to X$ in~$\cat G$. Suppose that each $\alpha^*$ has a right adjoint $\alpha_*:\cat D(Y)\to \cat D(X)$. Then, we say that $\cat D$ has the \emph{Beck-Chevalley property} if for every pull-back square
$$
\xymatrix@R=2em{
Y' \ar[r]^-{\beta'} \ar[d]_-{\alpha'}
& Y \ar[d]^-{\alpha}
\\ X' \ar[r]^-{\beta}
& X
}
$$
in~$\cat G$, we have a base-change formula, $\beta^*\alpha_* \simeq \alpha'_*{\beta'}^*$, more precisely the morphism
\begin{equation}
\label{eq:BC}%
\xymatrix@C=4em{
\beta^*\alpha_* \ar[r]^-{\eta\pp{\alpha'}}
& \alpha'_*{\alpha'}^*\beta^*\alpha_* \ar@{=}[r]^-{(\beta\alpha'=\alpha\beta')}
& \alpha'_*{\beta'}^*\alpha^*\alpha_* \ar[r]^-{\eps\pp{\alpha}}
& \alpha'_*{\beta'}^*}
\end{equation}
is an isomorphism. (We use that $\cat D(-)$ is a strict functor but again the notion makes sense for pseudo-functors, replacing the middle identity by an isomorphism.)
\end{Def}

\begin{Thm}[B\'enabou-Roubaud~\cite{BenabouRoubaud70}]
\label{thm:BR}%
Let $\cat G$ be a site with pull-backs and let $\cat D:\cat G\op\to \Cat$ be a functor with the Beck-Chevalley property. Let $\alpha:U\to X$ be a cover. The adjunction $\alpha^*:\cat D(X)\adj \cat D(U):\alpha_*$ defines a comonad $L:=\alpha^*\alpha_*\,:\ \cat D(U)\to \cat D(U)$ and we can compare $\cat D(X)$ with the category of $L$-comodules in~$\cat D(U)$, via an Eilenberg-Moore functor~$E$ as in Remark~\ref{rem:KE} (for the dual)\,:
$$
\xymatrix@C=4em@R=.7em{
\cat D(X) \quad \ar@<-.3em>[rd]_(.5){\Displ\alpha^*} \ar@<.3em>[rr]^-{\Displ E} \ar@/_1.5em/[rddd]_-{\Displ Q}
&& \quad L\CComod_{\cat D(U)} \ar@<-.5em>[ld]_(.7){\textrm{forget}}
\\
& \ \cat D(U) \ \ar@<-.3em>[lu]_(.3){\Displ\alpha_*} \ar[ru]_(.4){\textrm{free}}
  \ar@(ld,rd)[]_(.22){L}
\\
\\
& \ \Desc_{\cat D}(U\oto{\alpha}X) \ar@{-->}@/_1.5em/[ruuu]^-{\simeq}_-{\Displ \exists\,B}
}
$$
Then there exists an equivalence $B:\Desc_{\cat D}(U\oto{\alpha}X) \isotoo L\CComod_{\cat D(U)}$ such that $B\circ Q\cong E$. Consequently $\cat D$ satisfies descent with respect to $\alpha:U\to X$ ($Q$ is an equivalence) if and only if the adjunction $\alpha^*/\alpha_*$ is \emph{comonadic} ($E$ is an equivalence).
\end{Thm}

\begin{Rem}
We shall not prove this classical result but, since \cite{BenabouRoubaud70} gives little detail, we quickly indicate why this holds. Consider the pull-back square
$$
\xymatrix@C=3em@R=2em{
\kern-2em U\pp{2}=U\times_X U \ar[r]^-{\pr_2} \ar[d]_-{\pr_1}
& U \ar[d]^-{\alpha}
\\
U \ar[r]^-{\alpha} & X
}
$$
and consider a gluing datum $(W,s)$ in the descent category $\Desc_{\cat D}(U\oto{\alpha}X)$ as in Definition~\ref{def:descent}. The gluing isomorphism $s:\pr_2^*W\isoto \pr_1^*W$ in $\cat D(U\pp{2})$ defines, by the $\pr_2{}^*/\pr_2{}_*$ adjunction, a morphism $W\to (\pr_2)_*\pr_1^*W$. By the Beck-Chevalley property applied to the above pull-back, we have $(\pr_2)_*\pr_1^*=\alpha^*\alpha_*$ which is the comonad~$L$. This new morphism $W\to L(W)$ makes $W$ into an $L$-comodule and this assignment yields the functor~$B$.
Note that the original source~\cite{BenabouRoubaud70} is stated dually, using the existence of \emph{left} adjoints to~$\alpha^*$ (somewhat unfortunately denoted $\alpha_*$ instead of the now common $\alpha_{\,!}$) and monads instead of comonads. Of course, our statement is a formal consequence of that one, via opposite categories.
\end{Rem}

We can now use the above technique to prove the fundamental result of the paper. We denote by $\bbZ_{(p)}=\SET{\frac{a}{b}\in\bbQ}{b\textrm{ is prime to }p}$ the local ring of~$\bbZ$ at~$p$.

\begin{Thm}
\label{thm:stack}%
Let $\cat A$ be an idempotent-complete additive category over $\bbZ_{(p)}$, \eg\ $\cat A=\kk\MMod$ for $\kk$ a (commutative local ring with residue) field of characteristic~$p$. Then we have strict stacks $\Gsets\op\to\Add$ in the sense of Definition~\ref{def:stack}\,:
\begen%[(a)]
\item
The functor of plain representations $\Rep(-)=\cat A^{G\ltimes-}$ of Definition~\ref{def:rep} is a strict stack on $\Gsets$ for the sipp topology (Definition~\ref{def:top}).
\smallbreak
\item
If $\cat A$ is moreover abelian, then the functor of derived categories $\Der(\Rep(-))$ is a strict stack on~$\Gsets$ for the sipp topology.
\smallbreak
\item
If $\cat A=\kk\VVect$ for a field $\kk$ of characteristic~$p$, then the functor of stable categories $\StRep(-)$ is a strict stack on~$\Gsets$ for the sipp topology.
\ened
\end{Thm}

\begin{proof}
The strategy of the proof is the following. First, in Lemma~\ref{lem:BC}, we prove the Beck-Chevalley property. This reduces descent to comonadicity, by Theorem~\ref{thm:BR}. Then, we prove comonadicity in Lemma~\ref{lem:comod}.

\begin{Lem}
\label{lem:BC}%
The functor $\Rep(-):\Gsets\op\to \Add$ of plain representations satisfies the Beck-Chevalley property of Definition~\ref{def:BC}. So does the derived one $\Der(\Rep(-)):\Gsets\op\to \Add$ when $\cat A$ is abelian and the stable one $\StRep(-):\Gsets\op\to \Add$ when $\cat A=\kk\VVect$ for $\kk$ a field.
\end{Lem}

\begin{proof}
We start with the plain one $\Rep(-)$. Consider a pull-back in $\Gsets$\,:
$$
\xymatrix@C=2em@R=1.5em{
Y' \ar[r]^-{\beta'} \ar[d]_-{\alpha'}
& Y \ar[d]^-{\alpha}
\\
X' \ar[r]^-{\beta} & X
}
$$
Note that being a pull-back implies that for every $x'\in X'$ we have a bijection
\begin{equation}
\label{eq:BC-trick}%
(\alpha')\inv(x')\isoto \alpha\inv(\beta\,x')\qquadtext{given by}y'\mapsto \beta'y'.
\end{equation}
We need to check that the morphism~\eqref{eq:BC} $\beta^*\alpha_*V\to \alpha'_*{\beta'}^*V$ is an isomorphism over~$X'$ for every representation $V$ over~$Y$. Unfolding the definitions of $\eta\pp{\alpha'}$ and of $\eps\pp{\alpha}$ given in Proposition~\ref{prop:adj}, we obtain for every $x'\in X'$ the morphisms
$$
\xymatrix@R=.5em{
(\beta^*\alpha_* V)_{x'}
 \ar[r]^-{\eta\pp{\alpha'}} \ar@{=}[d]
& (\alpha'_*{\alpha'}^*\beta^*\alpha_* V)_{x'}=(\alpha'_*{\beta'}^*\alpha^*\alpha_*  V)_{x'}
 \ar[r]^-{\eps\pp{\alpha}_{}} \ar@{=}[d]
& (\alpha'_*{\beta'}^* V)_{x'} \ar@{=}[d]
\\
\Displ\prod_{y\in\alpha\inv(\beta\,x')} V_y
& \Displ\prod_{y'\in {\alpha'}\inv(x')}\ \prod_{y\in\alpha\inv(\beta x')}V_y
& \Displ\prod_{y'\in {\alpha'}\inv(x')}V_{\beta'(y')}
\\
(v_y)_y \ar@{|->}[r]
& (v_y)_{y',y}
\quad,\quad (v_{y',y})_{y',y} \ar@{|->}[r]
& (v_{y',\beta'(y')})_{y'}
\\
(v_y)_y \ar@{|->}[rr]
&& (v_{\beta'(y')})_{y'}}
$$
and this composition is indeed an isomorphism by the bijection~\eqref{eq:BC-trick}. For the derived and stable ones, just observe that the units and counits are defined (degreewise) by the plain ones and a plain isomorphism trivially remains an isomorphism in the derived and stable categories. See Lemma~\ref{lem:der}.
\end{proof}

\begin{Lem}
\label{lem:unit}%
Let $F:\cat D\adj\cat D:G$ be an adjunction of idempotent-complete additive categories. Suppose that the unit $\eta: \Idcat{C}\to GF$ has a natural retraction. Then the adjunction is comonadic.
\end{Lem}

\begin{proof}
This is the dual of~Lemma~\ref{lem:monadicity}\,(c).
\end{proof}

So far, we did not use the assumptions about the prime~$p$ but here it comes\,:

\begin{Lem}
\label{lem:comod}%
Let $\alpha:U\to X$ be a cover in the sipp topology (Definition~\ref{def:top}). Then the adjunction $\alpha^*:\Rep(X)\adj \Rep(U):\alpha_*$ is comonadic. Again, the same is true for the derived and stable adjunctions, when they make sense.
\end{Lem}

\begin{proof}
Since $\cat A$ is idempotent-complete then so are $\Rep(X)=\cat A^{G\ltimes X}$ and $\Rep(U)=\cat A^{G\ltimes U}$. To apply Lemma~\ref{lem:unit}, we claim that the unit $\eta\pp{\alpha}:\Id_{\Rep(X)}\to\alpha_*\alpha^*$ admits a natural retraction, that is, there exists a natural transformation $\pi:\alpha_*\alpha^*\to \Id_{\Rep(X)}$ such that $\pi\circ\eta\pp{\alpha}=\id$. By additivity in the base~$X$, \ie $\Rep(X_1\sqcup X_2)=\Rep(X_1)\sqcup \Rep(X_2)$, we can assume that $X$ is an orbit, say $X=G/H$. By additivity in~$U$ and the fact that a natural transformation $\smallmatrice{\eta_1\\\vdots \\\eta_n}:\Id\to F_1\oplus F_2\oplus\cdots F_n$ is retracted as soon as one of the~$\eta_i$ is, it suffices to show the claim for the restriction of $\alpha$ to \emph{some} orbit of~$U$. Since $U\to X=G/H$ is a sipp-cover, there is one orbit of~$U$ whose stabilizer has index prime to~$p$ in~$H$, so we choose that one. We are now reduced to the case where $\alpha:G/K\to G/H$ is the projection associated to a subgroup $K\leq H$ of index prime to~$p$. Note that for every $x\in X=G/H$, we now have $|\alpha\inv(x)|=[H:K]$ which is prime to~$p$ hence invertible in~$\cat A$. We can therefore define $\pi_V:\alpha_*\alpha^*V\to V$ for ever $V\in\Rep(X)$ by the formula
$$
(\pi_V)_x\,:\ (\alpha_*\alpha^* V)_x=\prod_{u\in\alpha\inv(x)}\kern-.5em V_{x}\ \too \ V_x
\qquad\
(v_u)_{u\in\alpha\inv(x)}\mapsto \frac{1}{[H:K]}\sum_{u\in\alpha\inv(x)}\kern-1em v_u
$$
for every $x\in X$ and verify that this is a well-defined natural transformation with the wanted property $\pi\circ\eta\pp{\alpha}=\id$. The sum gives $G$-invariance of~$\pi_V$, using the bijection $g\cdot:\alpha\inv(x)\isoto \alpha\inv(gx)$. Hence the result for plain representations $\Rep(-)$.

For the derived version (recall~Remark~\ref{rem:warn}), we invoke Lemma~\ref{lem:der} to see that the derived and stable units are still retracted and then apply Lemma~\ref{lem:unit}.
\end{proof}

This finishes the proof of Theorem~\ref{thm:stack}, as explained before~\ref{lem:BC}.
\end{proof}

%------------------------------------------------------------------------------
%------------------------------------------------------------------------------
\goodbreak
\part{Applications}
\label{part:C}%
\bigbreak
%------------------------------------------------------------------------------
Recall that $G$ is assumed to be a finite group.
%------------------------------------------------------------------------------
\medbreak
\section{Taming Mackey formulas}
\label{se:tame}%
\medbreak
%------------------------------------------------------------------------------

Our Stack Theorem~\ref{thm:stack} for a sipp-cover $U\to X$ involves gluing isomorphisms over $U\pp 2$ and cocycle conditions over $U\pp 3$. Unfolding this data in the case of an elementary sipp-cover of the form $G/H\onto G/G$ for a subgroup $H$ of index prime to~$p$ hits the problems explained in Remark~\ref{rem:desc}, related to the Mackey formulas for $U\pp2=G/H\times G/H$ and the ``higher" Mackey formulas for $U\pp3=G/H\times G/H\times G/H$. Our strategy around this problem is to study $U\pp2$ by accepting \emph{all} intersections $H^g\cap H$ instead of just those for $g$ in a chosen set $S$ of representatives of~$\HGH$ and similarly for $U\pp3$. This creates excessive information which is harmless for $U\pp3$ and which can be trimmed for~$U\pp2$.

\begin{Not}
\label{not:bcd}%
Let $H,K\leq G$ be subgroups and let $g,g_1,g_2\in G$ be elements.
\begin{enumerate}[(a)]
\item
Suppose that $\con{g}K\leq H$. Consider the basic $G$-map, already used above
$$
\beta_g:G/K\to G/H
\qquadtext{given by}
[x]_K\mapsto [xg\inv]_H\,.
$$
Note that for $g=1$, \ie when $K\leq H$, the $G$-map $\beta_1$ is the projection $G/K\onto G/H$. There is a slight ambiguity  since notation $\beta_g$ does not display the subgroups $K$ and~$H$ but we will always make them clear in the sequel.
\smallbreak
\item
Mackey's formula~\eqref{eq:Mackey} involved $G$-maps that we now denote $\gamma_g:=\beta_g\times\beta_1$\,:
$$
\gamma_g\,:\ G/H[g]\too G/H\times G/H
\quadtext{,}
[x]_{H[g]}\mapsto ([xg\inv]_{H}\,,\,[x]_{H})
$$
using $\con{g}(H[g])\leq H$ and $\con{1}(H[g])\leq H$. Recall that $H[g]:=H^g\cap H$.
\smallbreak
\item
Recall that $H[g_2,g_1]:=H^{g_{2}g_{1}}\cap H^{g_1}\cap H$. Define $\delta_{g_2,g_1}:=\beta_{g_2g_1}\times \beta_{g_1}\times\beta_1$
$$
\begin{aligned}
\delta_{g_2,g_1}\,:\ G/\,H[g_2,g_1]
& \ \too \quad (G/H)\times (G/H)\times (G/H)
\\
[x]_{H[g_2,g_1]} & \ \longmapsto \ \big([x(g_2g_1)\inv]_{H}\,,\,[x\,g_1\inv]_{H}\,,\,[x]_{H}\big)
\end{aligned}
$$
using $\con{g_2g_1}(H[g_2,g_1])\leq H$ and $\con{g_1}(H[g_2,g_1])\leq H$ and $\con{1}(H[g_2,g_1])\leq H$.
\end{enumerate}
\end{Not}

\begin{Lem}
\label{lem:12}%
With the above notation, we have for every $g\in G\geq H$
\begin{equation}
\label{eq:12}%
\pr_1\circ \gamma_g=\beta_g\qquadtext{and}\pr_2\circ \gamma_g=\beta_1
\end{equation}
where $\pr_1,\pr_2\,:\ G/H\times G/H\to G/H$ are the projections. Also, for every $h\in H$
\begin{equation}
\label{eq:bc}%
\gamma_{hg}=\gamma_g\qquadtext{and}\gamma_{gh}=\gamma_{g}\,\beta_{h}
\end{equation}
as morphisms from $G/H[g]$ to $(G/H)^2$ and from $G/H[gh]$ to $(G/H)^2$, respectively.
\end{Lem}

\begin{proof}
Direct from the above definitions.
\end{proof}

\begin{Lem}
\label{lem:012}%
Let $H\leq G$ be a subgroup and $g_1,g_2\in G$. Then, using Notation~\ref{not:bcd}, the following three diagrams of $G$-sets commute ``separately" (\ie using on each side only the left, only the middle or only the right vertical maps, respectively)
$$
\xymatrix@C=3em@R=1.5em{
\kern3em G/\,H[g_2,g_1] \ar[r]^-{\Displ \delta_{g_2,g_1}}\kern3em
 \ar@<-4em>[d]_-{\Displ\beta_{g_1}} \ar[d]_-{\Displ\beta_{1}} \ar@<4em>[d]_-{\Displ\beta_{1}}
 \ar[d]^-{}
& (G/H) \times (G/H) \times (G/H)
 \ar@<-3.5em>[dd]_-{\Displ \pr_{12}} \ar[dd]_-{\Displ \pr_{13}} \ar@<3.5em>[dd]_-{\Displ \pr_{23}}
\\
G/H[g_2]\ ,\  G/H[g_2g_1]\ ,\ G/H[g_1]
 \ar@<-4em>[d]_-{\Displ\iota_{g_2}} \ar[d]_-{\Displ\iota_{g_2g_1}} \ar@<4em>[d]_-{\Displ\iota_{g_1}}
\\
\kern3em \coprod_{g\in G}\kern1em G/H[g]\kern3em
 \ar[r]^-{\coprod_{g}\Displ \gamma_g}
& \kern1em (G/H) \ \times \ (G/H)\,. \kern.5em
}
$$
Here $\iota_g$ denotes the inclusion into the term indexed by~$g\in G$.
\end{Lem}

\begin{proof}
For instance, the two compositions in the ``left-maps diagram" are\,:
$$
\xymatrix@R=1.5em{
[x]
 \ar@{|->}[r]^-{\delta_{g_2,g_1}} \ar@{|->}[d]^-{\iota_{g_2}\beta_{g_1}}
& \big([x\,(g_2g_1)\inv],[xg_1\inv],[x]\big)\kern-2em
 \ar@<2em>@{|->}[d]^-{\pr_{12}}
\\
[xg_1\inv]\textrm{ in term }H[g_2] \ar@{|->}[r]^-{\gamma_{g_2}}
& \big([xg_1\inv g_2\inv],[xg_1\inv]\big)=\big([x\,(g_2g_1)\inv],[xg_1\inv]\big)\,.
}
$$
The other two verifications are similarly direct from the definitions.
\end{proof}

We now unfold descent of Section~\ref{se:descent}. Note that this statement holds for any strict sipp-stack and avoids all non-canonical choices from Mackey formulas.

\begin{Thm}
\label{thm:tame}%
Let $\cat D:\Gsets\op\to \Add$ be a strict stack on~$\Gsets$ for the sipp topology (\eg\ those of Theorem~\ref{thm:stack}). Let $H\leq G$ be a subgroup of index prime to~$p$. Let $\alpha:G/H\onto G/G$ the associated $G$-map. Recall Notation~\ref{not:bcd}. Then, we have\,:
\begen
\smallbreak
\item
Let $V,V'\in\cat D(G/G)$. Then $f_0\mapsto \alpha^*(f_0)$ defines a bijection between the morphisms $f_0:V\to V'$ in $\cat D(G/G)$  and those morphisms $f:\alpha^*V\to \alpha^*V'$ in $\cat D(G/H)$ such that $\beta_1^*(f)=\beta_g^*(f)$ in $\cat D(G/H[g])$ for every $g\in G$.
\smallbreak
\item
Suppose given an object $W\in \cat D(G/H)$ and isomorphisms $s_g:\beta_1^*W\isoto\beta_g^*W$ in $\cat D(G/H[g])$ for every $g\in G$, for the $G$-maps $\beta_1,\beta_g:G/H[g]\to G/H$, and satisfying properties~(i) and~(ii) below\,:
\begen[(i)]
\smallbreak
\item
\label{it:H}%
For every $h\in H$ (in which case $H[h]=H$ and $\beta_1=\beta_h=\id_{G/H}$ and therefore $\beta_1^*W=W=\beta_h^*W$), assume that $s_h=\id_W$ in $\cat D(G/H)$\,.
\smallbreak
\item
\label{it:coc}%
For every $g_1,g_2\in G$, assume the following equality in $\cat D(G/\,H[g_2,g_1])$
$$
\beta_1^*(s_{g_2g_1})=\beta_{g_1}^*(s_{g_2})\circ \beta_1^*(s_{g_1})
$$
for $\beta_{1}:G/H[g_2,g_1]\to G/H[g_2g_1]$, $\beta_{g_1}:G/H[g_2,g_1]\to G/H[g_2]$ and $\beta_{1}:G/H[g_2,g_1]\to G/H[g_1]$ respectively.
\ened
Then there exists an object $V\in\cat D(G/G)$ and an isomorphism $f:\alpha^*V\isoto W$ in $\cat D(G/H)$ such that $\beta_g^*(f)=s_g\circ  \beta_1^*(f)$ in $\cat D(G/H[g])$ for every $g\in G$. Moreover, the pair $(V,f)$ is unique up to unique isomorphism of such pairs.
\ened
\end{Thm}

\begin{proof}
Since the index $\Index{G}{H}$ is prime to~$p$, we have a sipp-cover $U:=G/H\,\overset{\alpha}\onto$ $\, G/G=:X$. Since we assume that $\cat D$ is a stack, we have an equivalence of categories
$$
\cat D(X)\isotoo \Desc_{\cat D}(U\to X)\,.
$$
We want to describe the right-hand category. Recall from Definition~\ref{def:descent} that its objects are pairs $(W,s)$ where $W\in \cat D(U)=\cat D(G/H)$ and $s:\pr_2^*W\isoto\pr_1^*W$ is an isomorphism in $\cat D(U\pp2)=\cat D(G/H\times G/H)$ satisfying the cocycle relation
\begin{equation}
\label{eq:coc-UX}%
\pr_{13}^*(s)=\pr_{12}^*(s)\circ\pr_{23}^*(s)
\end{equation}
in $\cat D(U\pp3)=\cat D(G/H\times G/H\times G/H)$. Using Notation~\ref{not:bcd}, consider the functor
\begin{equation}
\label{eq:F2}%
F\pp2:=\prod_{g\in G}\gamma_g^*\,:\ \cat D\big((G/H)^2\big)\too \prod_{g\in G}\cat D(G/H[g])
\end{equation}
induced by all the $G$-maps $\gamma_g:G/H[g]\to (G/H)^2$. Similarly, the $G$-maps $\delta_{g_2,g_1}:G/\,H[g_2,g_1]\to (G/H)^3$ induce a functor that we denote
\begin{equation}
\label{eq:F3}%
F\pp3:=\prod_{g_2,g_1\in G}\delta_{g_2,g_1}^*\,:\ \cat D\big((G/H)^3\big)\too \prod_{g_2,g_1\in G}\cat D(G/\,H[g_2,g_1])\,.
\end{equation}
Choosing a representative set $S\subset G$ for~$\HGH$ and post-composing the functor~$F\pp2$ with the projection $\pr_S:\prod_{g\in G}...\too\prod_{g\in S}...$ we obtain an equivalence, by the Mackey formula~\eqref{eq:Mackey}. Similarly, if we choose moreover, for every $t\in S$, a representative set $S_t\subset G$ for $\doublequot{(H^t\cap H)}{G}{H}$ and if we post-compose $F\pp3$ with the projection $\prod_{g_2,g_1\in G}\,...\too\prod_{g_2\in S,\,g_1\in S_t}\,...$ we also obtain an equivalence, by two layers of Mackey formulas. In particular, both functors $F\pp2$ and $F\pp3$ are \emph{faithful}. For $F\pp2$ we can describe the image on morphisms more precisely (``trimming")\,:

\smallbreak
\noindent\textbf{Claim~A}\,: \textit{Given two objects $W'$ and $W''$ in the source category $\cat D((G/H)^2)$ of $F\pp2$, a morphism $(f_g)_{g\in G}$ from $F\pp2(W')$ to $F\pp2(W'')$ in the category $\prod_{g\in G}\cat D(G/H[g])$ belongs to the image of~$F\pp2$ if and only we have for every $h\in H$ and $g\in G$ that
\begin{align}
\label{eq:f_g}%
f_{hg}=f_{g}\qquadtext{and}f_{gh}=\beta_h^*(f_g)\,.
\end{align}}%
Here $\beta_h:G/H[gh]\to G/H[g]$ is as in Notation~\ref{not:bcd} using the relation $\con{h}(H[gh])\leq H[g]$. These conditions are necessary by~\eqref{eq:bc}. Conversely, assume that $(f_g)_{g\in G}$ satisfies~\eqref{eq:f_g}. Since the composition of $F\pp2$ with the projection onto those factors indexed by a representative set $S\subset G$ of~$\HGH$ is an equivalence (Mackey formula), we can find $f:W'\to W''$ in $\cat D((G/H)^2)$ such that at least
\begin{equation}
\label{eq:f_t}%
f_t=\gamma_t^*(f)\qquad\textrm{for all }t\in S\,.
\end{equation}
Let now $g\in G$ be arbitrary. We need to show that $f_g=\gamma_g^*(f)$ as well, which gives $F\pp2(f)=(f_g)_{g\in G}$. There exists $h_1,h_2\in H$ and $t\in S$ with $g=h_1\,th_2$ and then
$$
f_g=f_{th_2}=\beta_{h_2}^*(f_t)=\beta_{h_2}^*\gamma_t^*(f)=\gamma_{th_2}^*(f)=\gamma_g^*(f)
$$
using in turn\,: \eqref{eq:f_g}, \eqref{eq:f_t}, and the fact that $\gamma_{t}\beta_{h_2}=\gamma_{th_2}=\gamma_{h_1th_2}=\gamma_g$ by~\eqref{eq:bc} again. This proves Claim~A.

Let us prove~(a). The property that the functor $\cat D(X)\to \Desc_{\cat D}(U\oto{\alpha} X)$ is fully faithful means that for every $V,V'\in \cat D(X)$, the map $f_0\mapsto \alpha^*(f_0)$ is a bijection between $\Mor_{\cat D(X)}(V,V')$ and the set of those morphisms $f:\alpha^*V\to \alpha^*V'$ in $\cat D(G/H)$ such that $\pr_2^*(f)=\pr_1^*(f)$ in $\cat D((G/H)^2)$. Since $F\pp2$ is faithful, the later is equivalent to $F\pp2(\pr_2^*(f))=F\pp2(\pr_1^*(f))$. By~\eqref{eq:F2}, $F\pp2(\pr_2^*(f))=(\gamma_g^*\pr_2^*(f))_{g\in G}=(\beta_1^*(f))_{g\in G}$ using $\pr_2\gamma_g=\beta_1$ from~\eqref{eq:12}. Similarly, $F\pp2(\pr_1^*(f))=(\gamma_g^*\pr_1^*(f))_{g\in G}=(\beta_g^*(f))_{g\in G}$. Therefore $\pr_2^*(f)=\pr_1^*(f)$ if and only if $\beta_1^*(f)=\beta_g^*(f)$ for all $g\in G$, which is the condition of~(a).

Let us now prove the more juicy part~(b). Assume given the object $W\in\cat D(G/H)$ and the isomorphisms $s_g:\beta_1^*W\isoto \beta_g^*W$ in $\cat D(G/H[g])$, satisfying~(i) and~(ii). Uniqueness of $(V,f)$ up to unique isomorphism will follow from~(a), so we only need to prove existence of $V$ and $f:\alpha^*V\isoto W$ as announced.

\smallbreak
\noindent\textbf{Claim~B}\,: \textit{For every $h\in H$ and $g\in G$, we have
\begin{align}
\label{eq:claim2}%
& s_{hg}=s_{g} \qquadtext{ in }\cat D(G/H[hg])=\cat D(G/H[g])\qquad\textrm{and}
\\
\label{eq:claim1}%
& s_{gh}=\beta_h^*(s_{g}) \qquadtext{ in }\cat D(G/H[gh])\quadtext{for}\beta_h:G/H[gh]\to G/H[g]\,.
\end{align}
}%
To prove~\eqref{eq:claim1}, use condition~(ii) for $g_2=g$ and $g_1=h$, the fact that $s_h=\id$ by~(i) and finally that $\beta_1:G/\,H[g_2,g_1]\to G/H[g_2g_1]$ is the identity in that case. Similarly, \eqref{eq:claim2} follows from~(ii) for  $g_2=h$ and $g_1=g$, the same facts ($s_h=\id$ and $\beta_1=\id$) as above and the additional fact that $\beta_1: G/\,H[g_2,g_1]\to G/H[g_1]$ is also the identity in this case. This proves Claim~B.

Consider the two objects $W',W''\in\cat D((G/H)^2)$ given by $W'=\pr_2^*(W)$ and $W''=\pr_1^*(W)$. By~\eqref{eq:F2} and~\eqref{eq:12} we have $F\pp2(W')=(\gamma_g^*\pr_2^*W)_{g\in G}=(\beta_1^*W)_{g\in G}$ and $F\pp2(W'')=(\gamma_g^*\pr_1^*W)_{g\in G}=(\beta_g^*W)_{g\in G}$. By Claim~B, the morphism $(s_g)_{g\in G}$ from $F\pp2(W')$ to $F\pp2(W'')$ in $\prod_{g\in G}\cat D(G/H[g])$ satisfies conditions~\eqref{eq:f_g}. Then Claim~A gives the existence of a unique morphism $s:\pr_2^*W\to\pr_1^*W$ such that
\begin{equation}
\label{eq:s}%
\gamma_g^*(s)=s_g\qquadtext{in}\cat D(G/H[g])
\end{equation}
for every $g\in G$. This $s$ is necessarily an isomorphism by the same reasoning for the $s_g\inv$. We now need to prove that $s$ satisfies the cocycle condition~\eqref{eq:coc-UX} in $\cat D((G/H)^3)$. This can be tested by applying the faithful functor $F\pp3$ given in~\eqref{eq:F3}. On the other hand, it follows from Lemma~\ref{lem:012} and the above~\eqref{eq:s} that
\begin{align*}
& F\pp3(\pr_{12}^*(s))=\big(\beta_{g_1}^*(\gamma_{g_2}^*(s))\big)_{g_2,g_1}=\big(\beta_{g_1}^*(s_{g_2})\big)_{g_2,g_1}
\\
& F\pp3(\pr_{13}^*(s))=\big(\beta_{1}^*(\gamma_{g_2g_1}^*(s))\big)_{g_2,g_1}=\big(\beta_{1}^*(s_{g_2g_1})\big)_{g_2,g_1}
\\
& F\pp3(\pr_{23}^*(s))=\big(\beta_{1}^*(\gamma_{g_1}^*)\big)_{g_2,g_1}=\big(\beta_{1}^*(s_{g_1})\big)_{g_2,g_1}\,.
\end{align*}
Hence the image by the faithful functor $F\pp3$ of the cocycle condition~\eqref{eq:coc-UX} $\pr_{13}^*(s)=\pr_{12}^*(s)\circ\pr_{23}^*(s)$ becomes $\beta_1^*(s_{g_2g_1})=\beta_{g_1}^*(s_{g_2})\circ \beta_1^*(s_{g_1})$ for every $g_1,g_2\in G$ and this is precisely condition~(ii). By essential surjectivity of $\cat D(X)\to \Desc_{\cat D}(U\oto{\alpha} X)$, there exists $V\in\cat D(X)$ and an isomorphism~$f:\alpha^*V\isoto W$ in~$\cat D(U)$ such that $s\circ \pr_2^*f=\pr_1^*f$ in $\cat D(U\pp2)$. A last application of $F\pp2$ together with~\eqref{eq:12} and~\eqref{eq:s} turns this last equality into the wanted $s_g\circ\beta_1^*(f)=\beta_g^*(f)$ for all~$g\in G$.
\end{proof}

%------------------------------------------------------------------------------
\medbreak
\section{Extending modular representations}
\label{se:extend}%
\medbreak
%------------------------------------------------------------------------------

Let $\kk$ be a local commutative ring over $\bbZ_{(p)}$ and $\cat A=\kk\MMod$. When dealing with stable categories, we assume \emph{without further mention} that $\kk$ is a field.

\smallbreak

Recall Notations~\ref{not:123} and~\ref{not:123'}\,:
\begin{Not}
\label{not:Cs}
Let $\cat D:\Gsets\op\too\Add$ be any of the three sipp-stacks that we have considered in Section~\ref{se:descent} and correspondingly for $\cat C(H)$ when $H\leq G$ is a subgroup\,:
\begin{enumerate}[(1)]
\item
either plain categories $\cat D(X)=\Rep(X)=\kk\MMod^{G\ltimes X}$ and $\cat C(H)=\kk H\MMod$,
\item
or derived categories $\cat D(X)=\Der(\Rep(X))$ and $\cat C(H)=\Der(\kk H)$,
\item
or stable categories $\cat D(X)=\StRep(X)$ and $\cat C(H)=\kk H\SStab$.
\end{enumerate}
\end{Not}

\begin{Rem}
Recall from Example~\ref{exa:BH} the equivalence of groupoids $\iota_H:\rmB H\isoto G\ltimes(G/H)$, $\ast\mapsto[1]_H$, which induces the equivalence on ``plain" representations $\iota_H^*:\Rep(G/H)\isotoo \kk H\MMod$ of~\eqref{eq:iota}. Like every equivalence, $\iota_H^*$ is exact and preserves projective objects, hence it induces equivalences on derived and stable categories, that we still denote $\iota_H^*$. In short, we have for every subgroup $H\leq G$
\begin{equation}
\label{eq:iota-Cs}%
\iota_H^*:\cat D(G/H)\isotoo \cat C(H)\,.
\end{equation}
We now want to transpose some of the functorial behavior of~$\cat D(-)$ to $\cat C(-)$.
\end{Rem}

\begin{Not}
\label{not:gRes}%
Let $H,K\leq G$ be subgroups and let $g\in G$ such that $\con{g}K\leq H$, then every $\kk H$-module $W$ has a \emph{$g$-twisted restriction} from~$H$ to~$K$, denoted $\cRes{g}^H_K(W)$ (or $\con{g\,}W\dto_K$ for brevity, as in the Introduction), which is defined as the same $\kk$-module $W$ but with $K$-action $k\cdot w:=\con{g}k\,w$. This is restriction along the group monomorphism $\con{g}(-):K\to H$ given by conjugation. Note that when $h\in H$ then $W\otoo{h\cdot} W$, $w\mapsto h\,w$, defines a natural isomorphism of $\kk K$-modules that we denote
\begin{equation}
\label{eq:tau}%
\tau_h\,:\ \Res^H_K \,W\isotoo\cRes{h}^H_K \,W\,.
\end{equation}
Note also that $\cRes{g_2g_1}=\cRes{g_1}\circ\cRes{g_2}$, which expresses contravariance of restriction together with $\con{g_2g_1}(-)=\con{g_2}(-)\circ\con{g_1}(-)$. These structures $\gRes^H_K$ and $\tau_h$ pass to the derived categories $\Der(\kk ?\MMod)$ degreewise and to the stable categories $\kk ?\SStab$.
\end{Not}

\begin{Lem}
\label{lem:BH}%
Let $H\leq G$. Let $K_1\leq G$ be a subgroup and $x_1\in G$ with $\con{x_1}K_1\leq H$.
\begen[(a)]
\item
Let $\beta_{x_1}:G/K_1\to G/H$ be the $G$-map of Notation~\ref{not:bcd}, given by $\beta_{x_1}([g]_{K_1})=[g\,x_1\inv]_H$. Then in diagram~\eqref{eq:bi} below, the left-hand square commutes up to the isomorphism $\omega\pp{x_1}:\iota_{K_1}^*\beta_{x_1}^*\isoto \cRes{x_1}^H_{K_1}\iota_H^*$ given for each $V\in \Rep(G/H)$ by $\omega\pp{x_1}_V:=V_{x_1}\,:\ V_{[x_1\inv]}\isoto V_{[1]}$. Similarly for derived and stable categories, where the natural transformation $\omega\pp{x_1}$ is defined objectwise.
\begin{equation}
\label{eq:bi}%
\vcenter{\xymatrix@C=5em{
\cat D(G/H) \ar[r]^-{\Displ(\beta_{x_1})^*} \ar[d]_-{\Displ \iota_H^*}
& \cat D(G/K_1) \ar[d]^-{\Displ \iota_{K_1}^*} \ar[r]^-{\Displ(\beta_{x_2})^*} \ar@{ :>}[ld]^-{\omega\pp{x_1}}_-{\simeq}
& \cat D(G/K_2) \ar[d]^-{\Displ \iota_{K_2}^*} \ar@{ :>}[ld]^-{\omega\pp{x_2}}
\\
\cat C(H) \ar[r]_-{\Displ\cRes{x_1}^H_{K_1}}
& \cat C(K_1) \ar[r]_-{\Displ\cRes{x_2}^{K_1}_{K_2}}
& \cat C(K_2)
}}
\end{equation}
\item
If $K_2\leq G$ is another subgroup and $x_2\in G$ is such that $\con{x_2}K_2\leq K_1$, then, considering the right-hand square in the above diagram, the horizontal composition of the natural transformations $\omega\pp{x_1}\circ\omega\pp{x_2}$ coincides with $\omega\pp{x_1x_2}$, under the identities $\beta_{x_2}^*\circ \beta_{x_1}^*=\beta_{x_1x_2}^*$ and $\cRes{x_2}^{K_1}_{K_2}\circ \cRes{x_1}^H_{K_1}=\cRes{x_1x_2}^H_{K_2}$.
\ened
\end{Lem}

\begin{proof}
Let $V\in \Rep(G/H)$. Compute $\iota_{K_1}^*(\beta_{x_1})^*V=V_{\beta_{x_1}[1]_{K_1}}=V_{[x_1\inv]_H}$ with action of $k\in K_1$ given by $V_k$, using that $k\in H^{x_1}=\Aut_{G/H}([x_1\inv]_H)$. On the other hand, $\cRes{x_1}^H_{K_1}\iota_H^*V=V_{[1]_H}$ with action of $k\in K_1$ given by $V_{\con{x_1}k}$\,, using that $\con{x_1}k\in H=\Aut_{G/H}([1]_H)$. Is is now a direct computation to see that $\omega\pp{x_1}=V_{x_1}:V_{[x_1\inv]}\isoto V_{[1]}$ is $K_1$-linear. The second part is a direct verification. Recall that the horizontal composition $\omega\pp{x_1}\circ\omega\pp{x_2}$ is defined as the composition
$$
\xymatrix@C=5em{
\iota_{K_2}^*\circ\beta_{x_2}^*\circ\beta_{x_1}^* \ar[r]^-{{\Displ\omega\pp{x_2}}\beta_{x_1}^*}
& \cRes{x_2}^{K_1}_{K_2}\circ\iota_{K_1}^*\circ\beta_{x_1}^* \ar[r]^-{\cRes{x_2}^{K_1}_{K_2}\,{\Displ \omega\pp{x_1}}}
& \cRes{x_2}^{K_1}_{K_2}\circ \cRes{x_1}^H_{K_1}\circ\iota_H^*
}
$$
and use that $V_{x_1}\circ V_{x_2}=V_{x_1x_2}$. Finally, all functors in sight are exact and preserve projective objects, so they pass to derived and stable categories on the nose (in particular without deriving them in the former case). See Lemma~\ref{lem:der}.
\end{proof}

\begin{Rem}
\label{rem:BH}%
Lemma~\ref{lem:BH}\,(a) for $x_1=1$ gives in particular an \emph{equality} $\Res^H_K\circ\iota_H^*=\iota_K^*\beta_1^*$ when $K\leq H$ and $\beta_1:G/K\onto G/H$ is the projection, for then $\omega\pp1=\id$.
\end{Rem}

We are now ready to prove Theorem~\ref{thm:pub}, which we state in more general form\,:

\begin{Thm}
\label{thm:stack0}%
Let $\cat C(-)$ be as in Notation~\ref{not:Cs}. Let $H\leq G$ be a subgroup of index prime to~$p$. Recall that $H[g]=H^g\cap H$ and that $H[g_2,g_1]=H^{g_2g_1}\cap H^{g_1}\cap H$.
\begen[(A)]
\item
Let $V',V''\in \cat C(G)$. Then $f_0\mapsto \Res^G_H(f_0)$ induces a bijection between the morphisms $f_0:V'\to V''$ in $\cat C(G)$ and those morphisms $f:\Res^G_H(V')\to \Res^G_H(V'')$ such that $\gRes^H_{H[g]}(f)\circ\tau_g=\tau_g\circ\Res^H_{H[g]}(f)$ in $\cat C(H[g])$ for all $g\in G$.
\smallbreak
\item
Suppose given an object $W\in \cat C(H)$ and for every $g\in G$ an isomorphism $\sigma_g:\Res^H_{H[g]}W\isoto \cRes{g}^H_{H[g]}W$ in $\cat C(H[g])$. Suppose further that\,:
\begen
\item[(I)]
For every $h\in H$ we have $\sigma_h=\tau_h$ in $\cat C(H)$, where $\tau_h$ is as in~\eqref{eq:tau}.
\smallbreak
\item[(II)]
For every $g_1,g_2\in G$, the following diagram commutes in $\cat C({H[g_2,g_1]})$\,:
\begin{equation}
\label{eq:coc-stab}%
\vcenter{\xymatrix@C=4em{
& \Res^H_{{H[g_2,g_1]}} W
 \ar[ld]_-{\Displ \Res^{H[g_1]}_{{H[g_2,g_1]}}(\sigma_{g_1})\qquad}
 \ar[rd]^-{\Displ \qquad\Res^{H[g_2g_1]}_{{H[g_2,g_1]}}(\sigma_{g_2g_1})}
\\
\cRes{g_1}^H_{{H[g_2,g_1]}} W\
 \ar[rr]_-{\Displ \cRes{g_1}^{H[g_2]}_{{H[g_2,g_1]}} (\sigma_{g_2})}
&& \ \cRes{g_2g_1}^H_{{H[g_2,g_1]}} W\,.\kern-3em
}}
\end{equation}
\ened
Then there exists an object $V\in \cat C(G)$ with an isomorphism $f:\Res^G_H V\isoto W$ in $\cat C(H)$ such that the following square commutes in $\cat C(H[g])$ for every $g\in G$\,:
$$
\xymatrix@C=6em{
\Res^G_{H[g]} V \ar[r]^-{\Displ \Res^H_{H[g]} f}_-{\simeq} \ar[d]_-{\Displ\tau_g}^-{\simeq}
& \Res^H_{H[g]} W \ar[d]^-{\Displ\sigma_g}_-{\simeq}
\\
\gRes^G_{H[g]} V \ar[r]^-{\Displ \gRes^H_{H[g]} f}_-{\simeq}
& \gRes^H_{H[g]} W\,.\kern-1em
}
$$
Moreover, the pair formed by the object $V$ in $\cat C(G)$ and the isomorphism $f$ in $\cat C(H)$ is unique up to unique isomorphism of such pairs, in the obvious sense.
\ened
\end{Thm}

\begin{proof}
We ``push" Theorem~\ref{thm:tame} along the equivalences $\iota_K^*:\cat D(G/K)\isoto\cat C(K)$ of~\eqref{eq:iota-Cs} for all subgroups $K$ in sight. We leave (A) as an exercise and focus on~(B). First, the result is independent of $W$ up to isomorphism in~$\cat C(H)$. So, we can assume that $W=\iota_H^*\hat W$ for some $\hat W\in \cat D(G/H)$. Let $g\in G$. Consider the left-hand square below, coming from Lemma~\ref{lem:BH} applied with $K_1:=H[g]$ and $x_1=g$
$$
\xymatrix@C=1.5em@R=1.5em{
\hat W \ar@{}[r]|-{\in} \ar@{|->}[d]
& \cat D(G/H) \ar[rr]^-{\Displ(\beta_g)^*} \ar[d]_-{\Displ \iota_H^*}^-{\simeq}
&& \cat D(G/H[g]) \ar[d]^-{\Displ \iota_{H[g]}^*}_-{\simeq} \ar@{ :>}[lld]^-{\omega\pp{g}}_-{\simeq}
& \beta_g^*\hat W \ar@{}[l]|-{\ni} \ar@{|->}[d]
&& \beta_1^*\hat W \ar@{-->}[ll]_-{\simeq}^{\exists\,\Displ s_g}\ar@{|->}[d]
\\
W \ar@{}[r]|-{\in}
& \cat C(H)\ar[rr]_-{\Displ\gRes^H_{H[g]}}
&& \cat C(H[g])
& \iota_{H[g]}^*\beta_g^*\hat W \ar@{}[l]|-{\ni} \ar[r]^-{\simeq}_-{\omega\pp{g}_{\hat W}}
& \gRes^H_{H[g]} W
& \Res^H_{H[g]} W \ar[l]_-{\simeq}^{\sigma_g}
}
$$
involving the isomorphism $\omega\pp{g}:\iota_{H[g]}^*\circ\beta_g^*\isoto \gRes^H_{H[g]}\circ\iota_H^*$. By Remark~\ref{rem:BH}, we have $\iota_{H[g]}^*\circ \beta_1^*=\Res^H_{H[g]}\circ\iota_H^*$, hence $\iota_{H[g]}^*\circ \beta_1^*(\hat W)=\Res^H_{H[g]}W$. Since $\beta_1^*\hat W$ and $\beta_g^*\hat W$ are both in $\cat D(G/H[g])$ and since $\iota_{H[g]}^*$ is fully faithful, there exists an isomorphism $s_g:\beta_1^*\hat W\isoto \beta_g^*\hat W$ in $\cat D(G/H[g])$ such that $\iota_{H[g]}^*(s_{g})=(\omega\pp{g}_{\hat W})\inv\circ\sigma_{g}$, that is,
\begin{equation}
\label{eq:s_g}%
\omega\pp{g}_{\hat W}\circ\iota_{H[g]}^*(s_{g})=\sigma_{g}
\end{equation}
in $\cat C(H[g])$. We claim that the collection of $s_g:\beta_1^*\hat W\isoto \beta_g^*\hat W$, for all $g\in G$, satisfies conditions~(i) and~(ii) of Theorem~\ref{thm:tame}. For~(i), let $g=h\in H$. In that case $\beta_h=\id:G/H[h]\to G/H$ and $\omega\pp{h}_{\hat W}:\iota_H^*\hat W=W\to \cRes{h}^H_H\iota_H^*\hat W=\con{h\,}W$ is simply $\hat W_h=\tau_h$. By~(I), $\sigma_h=\tau_h$ as well. Hence $\iota_H^*(s_h)=(\omega\pp{h}_{\hat W})\inv\sigma_h=\tau_h\inv\tau_h=\id_W$ and $s_h=\id_{\hat W}$ as wanted. For~(ii), let $g_1,g_2\in G$ and let $K=H[g_2,g_1]$. We need to prove $\beta_1^*(s_{g_2g_1})=\beta_{g_1}^*(s_{g_2})\circ\beta_1^*(s_{g_1})$ in $\cat D(G/K)$. Since $\iota_K^*:\cat D(G/K)\isoto \cat C(K)$ is faithful, it suffices to prove
$$
\iota_K^*\circ\beta_1^*(s_{g_2g_1})=\big(\iota_K^*\circ\beta_{g_1}^*(s_{g_2})\big)\circ\big(\iota_K^*\circ\beta_1^*(s_{g_1})\big)
$$
in $\cat C(K)$. This is the outer commutativity of the following diagram in~$\cat C(K)$\,:
$$
\xymatrix@C=1.3em@R=3em{
&& \iota_K^*\beta_1^*\hat W
 \ar@{=}[d]
 \ar@<.5em>[rrd]^(.45){\quad\Displ\iota_K^*\beta_1^*(s_{g_2g_1})}
 \ar@<-.5em>[lld]_(.45){\Displ\iota_K^*\beta_1^*(s_{g_1})\quad}
\\
\iota_K^*\beta_1^*\beta_{g_1}^*\hat W \kern-2em
 \ar@{=}[ddd]
&& \Res^H_K W
 \ar[ldd]^(.35){\kern-.5em\Res^{H[g_1]}_K(\sigma_{g_1})}
 \ar[rdd]_(.67){\Res^{H[g_2g_1]}_K(\sigma_{g_2g_1})\kern-.6em}
 \ar[ld]_-{\Res^{H[g_1]}_K\iota_{H[g_1]}^*(s_{g_1})\kern.5em}
 \ar[rd]^-{\qquad\Res^{H[g_2g_1]}_K\iota_{H[g_2g_1]}^*(s_{g_2g_1})}
 \ar@{}[dd]|(.6){\eqref{eq:coc-stab}}
&& \kern-2em \iota_K^*\beta_{g_1}^*\beta_{g_2}^*\hat W
 \ar@{=}[ddd]
\\
& \kern-2em\Res^{H[g_1]}_K\iota_{H[g_1]}^*\beta_{g_1}^*\hat W
 \ar@{=}[lu]
 \ar[d]_-{\Res^{H[g_1]}_K\omega\pp{g_1}_{\hat W}}
 \ar@{}[rd]_(.23){\eqref{eq:s_g}}%\textrm{ for }g_1\quad}
&& \kern1em \Res^{H[g_2g_1]}_K\iota_{H[g_2g_1]}^*\beta_{g_2g_1}^*\hat W \kern-2em
 \ar@{=}[ru]
 \ar[d]^-{\Res^{H[g_2g_1]}_K\omega\pp{g_2g_1}_{\hat W}}
 \ar@{}[ld]^(.23){\eqref{eq:s_g}}%\textrm{ for }g_2g_1}
&
\\
& \cRes{g_1}^H_K W\
 \ar[rr]^(.45){\cRes{g_1}^{H[g_2]}_K (\sigma_{g_2})}
 \ar[rd]|(.45){\cRes{g_1}^{H[g_2]}_K\iota_{H[g_2]}^*(s_{g_2})\qquad}
&& \ \cRes{g_2g_1}^H_K W
&
\\
\iota_K^*\beta_{g_1}^*\beta_1^*\hat W \kern-2em
 \ar@<-.5em>[rrd]_-{\Displ\iota_K^*\beta_{g_1}^*(s_{g_2})\ }
 \ar[ru]^-{\simeq}_-{\omega\pp{g_1}_{\beta_1^*\hat W}}
&& \cRes{g_1}^{H[g_2]}_{K}\beta_{g_2}^*\hat W
 \ar[ru]|(.5){\qquad \cRes{g_1}^{H[g_2]}_K\omega\pp{g_2}_{\hat W}}
 \ar@{}[ll]|{(\heartsuit)}
 \ar@{}[u]|(.6){\eqref{eq:s_g}}%|(.45){\textrm{for }g_2}
&& \kern-2em \iota_K^*\beta_{g_2g_1}^*\hat W
 \ar[lu]_-{\simeq}^-{\omega\pp{g_2g_1}_{\hat W}}
 \ar@{}[ll]|-{(\clubsuit)}
\\
&& \iota_K^*\beta_{g_1}^*\beta_{g_2}^*\hat W
 \ar@{=}@<-.5em>[rru]
 \ar[u]^-{\omega\pp{g_1}_{\beta_{g_2}^*\hat W}}_-{\simeq}
&}
$$
The ``key" triangles marked~\eqref{eq:s_g} commute by~\eqref{eq:s_g} applied to $g=g_1$, $g_2$ and $g_2g_1$ (anti-clockwise from left), to which we apply $\Res^{H[g_1]}_K$, $\cRes{g_1}^{H[g_2]}_K$, and $\Res^{H[g_2g_1]}_K$, respectively. The central triangle commutes by hypothesis~\eqref{eq:coc-stab} in~(II). The ``square" marked~$(\heartsuit)$ commutes by naturality of~$\omega\pp{g_1}_{\hat W}$ for the map~$s_{g_2}$. The ``square" marked~$(\clubsuit)$ commutes by part~(b) of Lemma~\ref{lem:BH}. The unmarked squares commute via easy identifications, as in Lemma~\ref{lem:BH}\,(b) but with $x_1$ or $x_2$ equal to~1.

We can therefore apply Theorem~\ref{thm:tame} to $\hat W$ and $(s_g)_{g\in G}$ to obtain an object $\hat V\in \cat D(G/G)$ and an isomorphism $\hat f:\alpha^*\hat V\isoto \hat W$ in $\cat D(G/H)$, where $\alpha:G/H\onto G/G$ is the sipp-cover. This gives an object $V:=\iota_G^*\hat V$ in $\cat C(G/G)$ and an isomorphism $f:=\iota_H^*\hat f:\Res^G_H V\isoto W$ in $\cat C(G/H)$, using that $\Res^G_H\circ\iota_G^*=\iota_H^*\circ\alpha^*$ by Remark~\ref{rem:BH}.
\end{proof}

\begin{Rem}
\label{rem:pedibus}%
Isomorphisms $\{\sigma_g\}_{g\in G}$ as in Theorem~\ref{thm:stack0} must exist for $W$ to extend to~$G$ (for $W=\Res^G_H(V)$, take $\sigma_g=g\cdot=\tau_g$). Along those lines, this result is not thrilling when $\cat C(?)=\kk?\MMod$ is the plain category of representations for, given $W\in \kk H\MMod$ and $(\sigma_g)_{g\in G}$ as in the statement, we can equip the ``underlying" $\kk$-module $V:=\Res^H_{1}W$ with the action $g\cdot v:=\sigma_g(v)$ for every $g\in G$. So, Theorem~\ref{thm:stack0} is particularly interesting for derived categories, where $\sigma_g$ is only an isomorphism in the derived category (\ie a fraction of quasi-isomorphisms) but even more so for stable categories as stated in Theorem~\ref{thm:pub} in the Introduction. Indeed, for stable categories, the trick of getting an ``underlying" object by restriction to the trivial subgroup breaks down completely since $\kk 1\SStab=0$. There is no ``fiber functor" on stable categories\,! The subtlety of the stable-category version resides in the vanishing of the stable category for groups of order prime to~$p$. In this vein, $\sigma_g$ is trivial when $H[g]$ has order prime to~$p$ and condition~(II) is void when $H[g_2,g_1]$ has order prime to~$p$. This illustrates how stable categories can actually be more flexible than plain ones. This flexibility may also be observed with $\otimes$-invertible objects\,: In $\kk G\MMod$, only one-dimensional representations are $\otimes$-invertible but in $\kk G\SStab$ there are much more $\otimes$-invertible (indeed, precisely the endotrivial $\kk G$-modules).
\end{Rem}

\begin{Rem}
In Theorem~\ref{thm:stack0}, if $W$ is finitely generated then so is~$V$. Similar statements apply with bounded complexes, etc. There are two proofs of this. Either note that the above proof works as soon as $\cat D(-)$ is idempotent-complete and preserved by the $\alpha^*/\alpha_*$ adjunction or prove directly, using that $\eta:\Id_{\cat C(G)}\to \CoInd_H^G\Res^G_H$ is split, that if $\Res^G_H V$ is finitely generated then so is~$V$.
\end{Rem}

\begin{Rem}
Let $H\leq G$ with index prime to~$p$. Suppose that $H$ is \emph{strongly $p$-embedded} in~$G$, in the sense that $H[g]$ has order prime to~$p$ whenever $g\in G$ is not in~$H$. Then it is well-known that $\Res^G_H:\kk G\SStab\isoto \kk H\SStab$ is an equivalence. This is compatible with Theorem~\ref{thm:stack0} because the isomorphisms $\sigma_g$ are completely forced in that case\,: (I) treats the case $g\in H$ and $\kk H[g]\SStab=0$ when $g\notin H$.
\end{Rem}

\begin{Rem}
\label{rem:normal}%
Let $H\unlhd G$ be a \emph{normal} subgroup of index prime to~$p$. Let $\Gamma=G/H$. Then Theorem~\ref{thm:stack0} is essentially formalizing the idea that $\kk G\SStab$ is the $\Gamma$-invariant part $(\kk H\SStab)^\Gamma$ of $\kk H\SStab$. This is compatible with the paradigm proposed in Theorem~\ref{thm:main}, where restriction $\kk G\SStab\too\kk H\SStab$ is viewed as a separable extension of scalars, for arbitrary subgroups $H\leq G$. In the normal case of $H\unlhd G$ this separable extension is moreover galoisian with Galois group $\Gamma=G/H$.
\end{Rem}

%------------------------------------------------------------------------------
\goodbreak\medbreak
\section{Endotrivial modules, from local to global}
\label{se:pic}%
\medbreak
%------------------------------------------------------------------------------

For this section, $\kk$ is a field of characteristic~$p$ dividing the (finite) order of~$G$.

\begin{Def}
\label{def:Cech}%
Let $\cat B$ be an abelian category (\eg\ $\cat B=\bbZ\MMod$). Let $P:\cat G\op\to \cat B$ be a $\cat B$-valued presheaf on a site~$\cat G$ and let $\cat U=\{U\to X\}$ be a cover in~$\cat G$. The \emph{\v Cech cohomology of~$\cat U$ with coefficients in~$P$}, denoted $\cH^\sbull(\cat U,P)\in\cat B$, is the cohomology of the following \emph{\Cech\ complex} $\cC^\sbull(\cat U,P)$ in~$\cat B$\,:
\begin{equation}
\label{eq:Cech}%
\xymatrix@C=1.5em{
0 \ar[r]
& P(U) \ar[r]^-{d^0}
& P(U\pp2) \ar[r]^-{d^{\,1}}
& \cdots P(U\pp{n}) \ar[r]^-{d^{n-1}}
& P(U\pp{n+1}) \ar[r]^-{d^n}
& P(U\pp{n+2}) \cdots
}\kern-1em
\end{equation}
with $\cC^n(\cat U,P)=P(U\pp{n+1})$ in degree~$n\geq 0$, with $U\pp{n}=U\times_X\cdots \times_XU$ ($n$ factors) and where the differential $d^n$ is the usual alternate sum of the maps induced by the $(n+2)$ projections $U\pp{n+2}\to U\pp{n+1}$. The map $P(X)\to P(U)$ induces an obvious map $P(X)\to \cH^0(\cat U,P)$ which is an isomorphism when $P$ is a sheaf.
\end{Def}

\begin{Exa}
Let $\Gm:\Gsets\op\to \Ab$ be the presheaf of abelian groups given by the stable automorphisms of the trivial representation, \ie for every $X\in\Gsets$
$$
\Gm(X)=\Aut_{\StRep(X)}(\unit)
$$
where $\unit\in\StRep(X)$ is given by $\unit_x=\kk$ for all $x\in X$ and $\unit_g=\id_\kk$ for every $g\in G$. Of course, for $Y\oto{\alpha} X$, $\Gm(\alpha)$ is the restriction~$\alpha^*$ using that $\alpha^*\unit=\unit$. When $X=G/H$ then $\iota_H^*(\unit)=\unit_{\kk H\sstab}$ is the usual $\otimes$-unit and its automorphism group is $\kk^\times$ when $H$ has order divisible by~$p$ and is trivial otherwise (for then $\kk H\sstab=0$). In other words, $\Gm=\underline{\kk^\times\!}$\, is the constant sipp-sheaf associated to the abelian group $\kk^\times$ as in Proposition~\ref{prop:uA}.
\end{Exa}

\begin{Rem}
For every $X\in \Gsets$, we can equip $\StRep(X)$ with a tensor by letting $(V\otimes V')_x:=V_x\otimes_{\kk}V'_x$ for every $x\in X$ and similarly for morphisms. The above representation $\unit$ is the $\otimes$-unit. We can consider the group of isomorphism classes of $\otimes$-invertible objects $\Pic(\StRep(X))$ which is an abelian group for~$\otimes$ as usual. For every $G$-map $\alpha:Y\to X$, the functor $\alpha^*:\StRep(X)\to \StRep(Y)$ is a $\otimes$-functor. It therefore induces a homomorphism $\alpha^*:\Pic(\StRep(X))\to\Pic(\StRep(Y))$. This yields a presheaf of abelian groups, simply denoted
$$
\StPic:\Gsets\op\to \Ab\,.
$$
When $X=G/H$, the equivalence $\iota_H^*:\StRep(G/H)\isoto \kk H\SStab$ is a (strict) $\otimes$-functor and yields an isomorphism $\StPic(G/H)\isoto \Pic(\kk H\sstab)=:T(H)$ with the group of isomorphism classes of $\otimes$-invertible objects in~$\kk H\sstab$, \aka the group of stable isomorphism classes of \emph{endotrivial} $\kk H$-modules. We denote by
$$
T(G,H):=\Ker\big(\Res^G_H\,:\ T(G)\too T(H)\big)
$$
the kernel of restriction. Whenever $\con{g}K\leq H$, the following diagram commutes
\begin{equation}
\label{eq:comp}%
\vcenter{\xymatrix@R=2em{
\StPic(G/H) \ar[r]^-{\simeq}_-{\iota_H^*} \ar[d]_-{\beta_g^*}
& T(H) \ar[d]^-{\gRes^H_K}
\\
\StPic(G/K) \ar[r]_-{\simeq}^-{\iota_K^*}
& T(K)
}}
\end{equation}
by Lemma~\ref{lem:BH}\,(a), where $\beta_g:G/K\to G/H$ is the usual map $[x]\mapsto [xg\inv]$ as in Notation~\ref{not:bcd}. In particular $\iota_G^*:\StPic(G/G)\isoto T(G)$ induces an isomorphism
$$
\Ker\big(\StPic(G/G)\otoo{\alpha^*} \StPic(G/H)\big)\isotoo T(G,H)
$$
where $\alpha=\beta_1\,:\ G/H\onto G/G$ is the only $G$-map.
\end{Rem}

\begin{Thm}
\label{thm:TGP}%
Let $H\leq G$ be a subgroup of index prime to~$p$ and let $\cat U$ be the associated sipp-cover $G/H\onto G/G$ in~$\Gsets$. Then there exists a canonical isomorphism $T(G,H)\cong\cH^1(\cat U,\Gm)$.
\end{Thm}

\begin{proof}
Let $\cat D=\StRep:\Gsets\op\to\Add$ be the sipp-stack of stable categories, established in Theorem~\ref{thm:stack}\,(c). It is a general fact for $\cat U=\{U\to X\}$ that $\cH^1(\cat U,\Gm)$ is isomorphic to the set of isomorphism classes of those $V\in\cat D(X)$ which become isomorphic to~$\unit$ in~$\cat D(U)$. Indeed, a 1-cocycle in the \Cech\ complex $\cC^\sbull(\cat U,\Gm)$ of~\eqref{eq:Cech} is nothing but a gluing isomorphism for the object $W=\unit\in\cat D(G/H)$ with respect to our cover~$\cat U$. Also note that $d^{\,0}$ is trivial in that case since $\Gm(\pr_1)=\Gm(\pr_2)$ for $\pr_i:U\pp2\to U$ the two projections; see Remark~\ref{rem:uA}.
\end{proof}

\begin{Thm}
\label{thm:H2}%
Let $H\leq G$ be a subgroup of index prime to~$p$. Let $U=G/H$ and consider $\cat U=\{U\,\overset{\alpha}\onto\, G/G\}$ the associated sipp-cover in $G$-sets. Then\,:
\begen[(a)]
\item
The image of the homomorphism $\alpha^*:\StPic(G/G)\to \StPic(U)$ is contained in $\cH^0(\cat U,\StPic)=\Ker\big(\pr_1^*-\pr_2^*:\StPic(U)\too\StPic(U\pp2)\big)$.
\smallbreak
\item
Let $w\in \cH^0(\cat U,\StPic)$. Choose $W\in\StRep(U)$ such that $w=[W]_{\simeq}$. Choose an isomorphism $\xi:\pr_2^*W\isoto \pr_1^*W$ in $\StRep(U\pp2)$. Define $\zeta:=\pr_{13}^*(\xi)\inv\circ\pr_{12}^*(\xi)\circ\pr_{23}^*(\xi)\in\Gm(U\pp3)$. Then $\zeta$ is a 2-cocycle in the \Cech\ complex $\cC^\sbull(\cat U,\Gm)$ of~\eqref{eq:Cech}. Moreover, the class of $\zeta$ in $\cH^2(\cat U,\Gm)$ only depends on~$w$ but neither on the choice of~$\xi$, nor on that of~$W$ as above.
\smallbreak
\item
Construction~(b) yields a well-defined group homomorphism
$$
z:\cH^0(\cat U,\StPic)\to \cH^2(\cat U,\Gm)\,,\qquad w\mapsto [\zeta(W,\xi)]\,.
$$
\smallbreak
\item
We have $\Img\big(\StPic(G/G) \otoo{\Displ \alpha^*}\StPic(U)\big)=
\Ker\big(\cH^0(\cat U,\StPic) \otoo{\Displ z} \cH^2(\cat U,\Gm)\big)$.
\smallbreak
\item
The isomorphism $\iota_H^*:\StPic(G/H)\isoto T(H)$ restricts to a canonical isomorphism
$\Ker\big(\cH^0(\cat U,\StPic) \oto{z} \cH^2(\cat U,\Gm)\big)\isotoo \Img\big(T(G)\to T(H)\big)$. \ened
\end{Thm}

\begin{proof}
We abbreviate $\cat D=\StRep$ the sipp-stack of Theorem~\ref{thm:stack}\,(c). Part~(a) is clear since $\StPic$ is a presheaf and $\alpha\circ\pr_1=\alpha\circ\pr_2$. Part~(b) contains a slight abuse of notation, since $\pr_{13}^*(\xi)\inv\circ\pr_{12}^*(\xi)\circ\pr_{23}^*(\xi)$ is an automorphism of~$\pr_3^*W$\,:
$$
\xymatrix@C=3em{
\pr_3^*W=\pr_{23}^*\pr_2^*W \ar[r]^-{\pr_{23}^*(\xi)} \ar@<-3em>[d]^-{\zeta\otimes\pr_3^*W}
& \pr_{23}^*\pr_1^*W=\pr_2^*W=\pr_{12}^*\pr_2^*W \ar@<3em>[d]^-{\pr_{12}^*(\xi)}
\\
\pr_3^*W=\pr_{13}^*\pr_2^*W \ar[r]^-{\pr_{13}^*(\xi)}
& \pr_{13}^*\pr_1^*W=\pr_1^*W=\pr_{12}^*\pr_1^*W
}
$$
where $\pr_3:U\pp3\to U$ is the projection on the third factor. But since $W$ is $\otimes$-invertible, so is $\pr_3^*W$. Therefore, like for every $\otimes$-invertible object, any automorphism of $\pr_3^*W$ is given by $\zeta\otimes \pr_3^*W$ for a unique automorphism $\zeta\in \Aut(\unit)=\Gm(U\pp3)$ of the unit. Note also that since $\otimes$ is symmetric the automorphisms of the unit act centrally on every morphism. The verification that $d^{\,2}(\zeta)=0$ is a direct computation. Now, if we replace $\xi$ by another isomorphism $\xi':\pr_2^*W\isoto \pr_1^*W$ then it will differ from $\xi$ by a central unit $u\in \Aut_{\cat D(U\pp2)}(\unit)=\Gm(U\pp2)$. It is then a direct computation to see that $\zeta(\xi)$ and $\zeta(\xi')$ differ by $d^{\,1}(u)$ in $\Gm(U\pp3)$, that is, the class of $\zeta$ is unchanged in $\cH^2(\cat U,\Gm)$. Finally, if $W'$ is another $\otimes$-invertible isomorphic to $W$ in $\cat D(U)$ then choose an isomorphism $v:W'\isoto W$ and use as $\xi'$ for $W'$ the composite $\pr_1^*(v)\inv\xi\pr_2^*(v)$. Then $v$ cancels out in the computation of $\zeta$, that is, $\zeta(\xi',W')=\zeta(\xi,W)$ . This proves~(b). Part~(c) is a standard exercise.

Part~(d) is where we use the stack property. The inclusion $\subseteq$ is easy. Indeed, if $W=\alpha^*V$ then one can take the identity as $\xi$, hence $\zeta$ is trivial. Conversely, if $[W]\in \cH^0(\cat U,\StPic)$ is such that, with the above notation, $[\zeta(W,\xi)]=0$ in $\cH^2(\cat U,\Gm)$ then $\zeta$ is a boundary, \ie $\zeta=d^{\,1}(u)$ for some $u\in \Gm(U\pp2)$. One can then modify the chosen isomorphism $\xi:\pr_2^*W\isoto \pr_1^*W$ into $u\otimes\xi$ (or $u\inv\otimes\xi$, depending on the sign convention in $d^{\,1}$) so that the new $\zeta(W,\xi)=1$. This means that this new $\xi$ satisfies the cocycle condition and therefore, $W$ can be descended to some $V\in \cat D(G/G)$ as wanted. One needs to check that $V$ is $\otimes$-invertible but this is formal\,: Either use that $W\potimes{-1}$ also descends or use that $\cat D(G/G)\to \cat D(U)$ is a faithful $\otimes$-exact functor between closed $\otimes$-triangulated categories, so it detects $\otimes$-invertibility.

Part~(e) follows by compatibility of $\StPic(-)$ and $T(-)$, see~\eqref{eq:comp} for $g=1$.
\end{proof}

Let us give the expected interpretation of~$\cH^0(\cat U,\StPic)$ in classical terms.

\begin{Prop}
\label{prop:H0}%
Let $H\leq G$ be a subgroup of index prime to~$p$ and let $\cat U=\{G/H\onto G/G\}$ be the associated sipp-cover. Then, under the isomorphism $\iota_H^*:\StPic(G/H)\isoto T(H)$, the subgroup $\cH^0(\cat U,\StPic)$ of $\StPic(G/H)$ becomes equal to $\SET{w\in T(H)}{\Res^H_{H[g]}(w)=\gRes^H_{H[g]}(w)\textrm{ in }T(H[g])\textrm{ for all }g\in G}$.
\end{Prop}

\begin{proof}
Choose $S\subset G$ a set of representatives $S\isoto\HGH$. Since for every $\kk H$-module $W$, we have $W\simeq \con{h\,}W$ for all $h\in H$, the above condition ``\,$\Res^H_{H[g]}(w)=\gRes^H_{H[g]}(w)$ for all $g\in G$\," is equivalent to the same condition for all $g\in S$ only. The result then follows from compatibility of $\StPic$ and $T$, as in~\eqref{eq:comp}, together with the isomorphism $\sqcup(\beta_g\times\beta_1)\,:\ \sqcup_{g\in S}\,G/H[g]\isoto (G/H)\pp{2}$ from~\eqref{eq:Mackey}.
\end{proof}

\begin{Exa}
Following up on Remark~\ref{rem:normal}, when $H\unlhd G$ is normal then $G/H$ acts on $T(H)$ and Proposition~\ref{prop:H0} gives $\cH^0(\cat U,\StPic)\cong T(H)^{G/H}$.
\end{Exa}

\begin{Rem}
The condition ``$w$ belongs to $\cH^0(\cat U,\Gm)$" is the \emph{naive} condition for extension to~$G$. Descent with respect to the sipp-cover $\cat U$ gives the critical obstruction $z(w)$ in $\cH^2(\cat U,\Gm)$, whose vanishing really guarantees extension to~$G$. If $H=P$ is the Sylow $p$-subgroup of~$G$, then all the subgroups $P[g]\leq P$ are $p$-groups and all $T(P[g])$ are known by the classification~\cite{CarlsonThevenaz04,CarlsonThevenaz05}. So Proposition~\ref{prop:H0} allows a complete description of the subgroup $\cH^0(\cat U,\StPic)$ of~$T(P)$.
\end{Rem}

\begin{Rem}
The \Cech\ cohomology groups $\cH^1(\cat U,\Gm)$ and $\cH^2(\cat U,\Gm)$ which appear in Theorems~\ref{thm:TGP} and~\ref{thm:H2} can be made more explicit. Indeed, for $i=1$, one exactly recovers the description of the kernel $T(G,H)$ is terms of so-called ``weak $H$-homomorphisms" $G\to \kk^\times$ as given in~\cite{Balmer12app}. The case of $\cH^2(\cat U,\Gm)$ is more technical but the complex $\cC^i(\cat U,\Gm)$ around $i=2$ only involves finitely many copies of $\kk^\times$, indexed by components of $U\pp{2}$, $U\pp{3}$ and $U\pp{4}$. This explicit description of $\cH^2(\cat U,\Gm)$ is left to the interested reader. In the author's opinion, it becomes preferable to use \Cech\ cohomology \textsl{per se} and not to obsess oneself with a direct computation from the definition. Instead, for specific groups $H\leq G$, one might try to use cohomological methods to compute $\cH^\sbull(\cat U,\Gm)$ or its torsion. Such developments are interesting challenges for future research.
\end{Rem}

%------------------------------------------------------------------------------
\noindent\textbf{Acknowledgments}\,: I thank Serge Bouc, Jon Carlson, Ivo Dell'Ambrogio, Rapha\"el Rouquier, Jacques Th\'evenaz and Peter Webb for several helpful discussions.
%------------------------------------------------------------------------------
\addtocontents{toc}{\protect\mbox{}\protect}
\bibliographystyle{alpha}%

%------------------------------------------------------------------------------

%------------------------------------------------------------------------------
\end{document}